\documentclass[11pt]{amsart}

\usepackage{slashbox,booktabs,amsmath}
\usepackage{amssymb}
\usepackage{amsthm}
\usepackage{graphics}
\usepackage{tikz}
\usepackage{tikz-qtree}
\usepackage{marvosym}
\usepackage{pgfplots}
\usepackage{amsfonts}
\usepackage{mathtools}
\usepackage{algorithm}
\usepackage{mathdots}
\usepackage[noend]{algpseudocode}
\usepackage{float}
\usepackage{caption}
\usepackage{rotating}
\usepackage{subfigure}
\usepackage{lscape}

\allowdisplaybreaks

\usepackage{color} %remove this later
\definecolor{ao(english)}{rgb}{0.0, 0.5, 0.0}
\newtheorem{definition}{Definition}
\newtheorem{lemma}{Lemma}
\newtheorem{theorem}{Theorem}
\newtheorem{proposition}{Proposition}
\newtheorem{corollary}{Corollary}

\newcommand*\circled[1]{\tikz[baseline=(char.base)]{
            \node[shape=circle,draw,inner sep=2pt] (char) {#1};}}

\newcommand{\rmv}[1]{}
\pgfplotsset{compat=1.14}

\begin{document}
\title{Maximal entries of elements in certain matrix monoids}
\author{Sandie Han, Ariane M. Masuda, Satyanand Singh, and Johann Thiel}
\date{\today}
\address{Department of Mathematics, New York City College of Technology \newline (CUNY), 300 Jay Street,
Brooklyn, New York 11201}
\email{\{shan,amasuda,ssingh,jthiel\}@citytech.cuny.edu}
\thanks{The second author received support for this project provided by a PSC-CUNY Award, \#69227-00 47, jointly funded by The Professional Staff Congress and The City University of New York.}
%%%%%%%%%%%%%%%%%%%%%%%%%%%%%%%%%%%%%%%%%%%%%%%%%%

\begin{abstract}
Let $L_u=\begin{bmatrix}1 & 0\\u & 1\end{bmatrix}$ and $R_v=\begin{bmatrix}1 & v\\0 & 1\end{bmatrix}$
be matrices in $SL_2(\mathbb Z)$ with $u, v\geq 1$. Since the monoid generated by $L_u$ and $R_v$ is free, we can associate a depth to each element based on its product representation. In the cases where $u=v=2$ and $u=v=3$, Bromberg, Shpilrain, and Vdovina found a depth $n$ matrix containing the maximal entry for each $n\geq 1$.  By using ideas from our previous work on positive linear fractional transformation $(u,v)$-Calkin-Wilf trees and a polynomial partial ordering, we extend their results for any $u, v\geq 1$ and in the process we recover the Fibonacci and some Lucas sequences. As a consequence we obtain bounds which guarantee collision resistance on a family of Cayley hash functions based on $L_u$ and $R_v$.
 \end{abstract}
\maketitle

%%%%%%%%%%%%%%%%%%%%%%%%%%%%%%%%%%%%%%%%%%%%%%%%%%

\section{Introduction}

Let $L_u:=\begin{bmatrix}1 & 0\\u & 1\end{bmatrix}$ and  $R_v:=\begin{bmatrix}1 & v\\0 & 1\end{bmatrix}$ where $u$ and $v$ are fixed positive integers.  These matrices have been considered in the study of growth and expansion in groups, expander graphs, and in connection with the girth of certain Cayley graphs in $SL_2(\mathbb{F}_p)$ ($p$ a prime) in important works by Bourgain and Gamburd~\cite{BG}, and Helfgott~\cite{H1}. We are interested in the monoid generated by $L_u$ and $R_v$.  Nathanson~\cite{N1} provided a simple proof that this monoid is free. That is, every non-identity element $M$ in the monoid generated by $L_u$ and $R_v$ can be written as an alternating product of positive powers of $L_u$ and $R_v$ in a unique way. We refer to the sum of these powers as the depth\footnote{Some readers may prefer the term length, but our choice of the word depth will be made clear shortly.} of $M$. For example, if $M=L_u^3R_v^{24}L_u^7$, then the depth of $M$ is $34$. We define the depth of the identity matrix
as zero.

In 2017, Bromberg, Shpilrain, and Vdovina~\cite{BSV} proposed a Cayley hash function\footnote{The term Cayley hash function comes from the connection to Cayley graphs~\cite{H} mentioned in the previous paragraph.} based on $L_u$ and $R_v$ in $SL_2(\mathbb{F}_p)$. Bromberg~\cite{B} and Bromberg, Shpilrain, and Vdovina~\cite{BSV} measured the collision resistance of these Cayley hash functions by finding a depth $n$ matrix containing the maximal entry in the cases where $u=v=2$ and $u=v=3$ for each $n\geq 1$. The proof is by induction. They show that if $M$ is a depth $n$ matrix containing the maximal entry and maximal column sum, then either $L_uM$ or $R_vM$, depending on the parity of $n$, must be a depth $n+1$ matrix containing the maximal entry and maximal column sum.

The goal of this paper is to answer and expand upon some open questions\footnote{The original questions in~\cite{B,BSV} only ask about the case $u=1$ and $v=2$ (or vice versa).} appearing in~\cite{B,BSV}. Our main contribution is to extend the above result to the general case $u,v\geq 1$ (Theorem~\ref{main}). In the case where $u,v\geq 2$, our method uses a similar induction argument as above. When either $u=1$ or $v=1$, this induction argument fails. We indicate at the end of Proposition~\ref{ineq} exactly where the problem occurs. In this case the situation is more complicated, requiring a different strategy. The novelty of our approach lies in the framework of the Calkin-Wilf tree. This tree provides a rich, highly symmetric structure that allows us to organize the elements of the monoid in an intuitive way. Together with a partial ordering for polynomials, this yields a proof for the more intricate case.

The Calkin-Wilf tree is a rooted tree whose vertices are labeled by positive rational numbers.  Starting from the root 1, each vertex $a/b$ has two children: the left one is $a/(a+b)$ and the right one is $(a+b)/b$ (see Figure~\ref{fig:CWtree}).  A nice feature of this tree is that every positive rational number appears in the tree exactly once and in reduced form. In past decades, the Calkin-Wilf tree has received a lot of attention due to its many remarkable properties; see~\cite{CW,N2} for a more thorough history of this material.

\begin{figure}[ht!]
\begin{center}
\begin{tikzpicture}[sibling distance=15pt]
\tikzset{level distance=30pt}
\Tree[.$1/1$ [.$1/2$ [.$1/3$ $1/4$ $4/3$ ]
   [.$3/2$ $3/5$ $5/2$ ] ] [.$2/1$ [.$2/3$  $2/5$ $5/3$ ] [.$3/1$ $3/4$ $4/1$ ] ]]
\end{tikzpicture}
\end{center}
\caption{The first four rows of the Calkin-Wilf tree.}\label{fig:CWtree}
\end{figure}

For a generalization of the Calkin-Wilf tree, one can consider a modified generation rule for each vertex $a/b$: its left child is defined by $a/(ua+b)$ and its right child by $(a+vb)/b$.  Such a tree is called the {\emph{$(u,v)$-Calkin-Wilf tree}} and its vertices are a subset of the vertices of the Calkin-Wilf tree. In fact, given  $uv>1$, the positive rationals can be partitioned in pairwise disjoint $(u,v)$-Calkin-Wilf trees with roots coming from a special ``orphan" set~\cite{HMST1,HMST2,N2}.

Another generalization of the Calkin-Wilf tree can be made for rooted trees whose vertices are matrices in $GL_2(\mathbb{N}_0)$ with a generation rule based on matrix multiplication with a pair of fixed matrices $L$ and $R$. Nathanson refers to such trees as generalized Calkin-Wilf trees for positive linear transformations\footnote{We also use the term PLFT here since there is a clear isomorphism between the monoid of PLFTs (under function composition) and $GL_2(\mathbb{N}_0)$. For a proof of this fact, see~\cite{N2} and~\cite{HMST2}.} (PLTFs)~\cite{N2}. In the case where $L=L_u$ and $R=R_v$, we construct the {\emph{PLFT $(u,v)$-Calkin-Wilf tree}}, denoted by $\mathcal{T}^{(u,v)}(M)$, according to the following generation rules:

\begin{enumerate}
\item the root is labeled $M$,
\item the left child of a vertex $\begin{bmatrix}a & b\\c & d\end{bmatrix}$ is labeled $\begin{bmatrix}a & b\\ua+c & ub+d\end{bmatrix}$, and
\item the right child of a vertex $\begin{bmatrix}a & b\\c & d\end{bmatrix}$ is labeled $\begin{bmatrix}a+vc & b+vd\\c & d\end{bmatrix}$.
\end{enumerate}

\begin{figure}[ht!]
\begin{center}
\begin{tikzpicture}[every tree node/.style={font=\tiny,anchor=base}, sibling distance=-8.5pt]
\tikzset{level distance=55pt}
\Tree[.${\begin{bmatrix} 1 & 0\\ 0 & 1\end{bmatrix}}$ \edge node[auto=right] {\tiny{$L_u$}}; [.${\begin{bmatrix} 1 & 0\\ u & 1\end{bmatrix}}$ \edge node[auto=right] {\tiny{$L_{u}^2$}}; [.${\begin{bmatrix} 1 & 0\\ 2u & 1\end{bmatrix}}$
\edge node[auto=right] {\tiny{$L_{u}^3$}};[.${\begin{bmatrix} 1 & 0\\ 3u & 1\end{bmatrix}}$ ]
\edge node[auto=left] {\tiny{$R_vL_{u}^2$}}; [.${\begin{bmatrix} 1+2uv & v\\ 2u & 1\end{bmatrix}}$ ] ]
\edge node[auto=left] {\tiny{$R_vL_u$}}; [.${\begin{bmatrix} 1+uv & v\\ u & 1\end{bmatrix}}$
\edge node[auto=right] {\tiny{$L_uR_vL_u$}}; [. ${\begin{bmatrix} 1+uv & v\\ u(2+uv) & 1+uv\end{bmatrix}}$ ]
\edge node[auto=left] {\tiny{$R_v^2L_u$}};   [.${\begin{bmatrix} 1+2uv & 2v\\ u & 1\end{bmatrix}}$ ] ]
    ] \edge node[auto=left] {\tiny{$R_v$}}; [.${\begin{bmatrix} 1 & v\\ 0 & 1\end{bmatrix}}$ \edge node[auto=right] {\tiny{$L_uR_v$}}; [.${\begin{bmatrix} 1 & v\\ u & 1+uv\end{bmatrix}}$
   \edge node[auto=right] {\tiny{$L_u^2R_v$}}; [.${\begin{bmatrix} 1 & v\\ 2u & 1+2uv\end{bmatrix}}$ ]
   \edge node[auto=left] {\tiny{$R_vL_uR_v$}}; [.${\begin{bmatrix} 1+uv & v(2+uv)\\ u & 1+uv\end{bmatrix}}$ ]  ]
  \edge node[auto=left] {\tiny{$R_v^2$}}; [.${\begin{bmatrix} 1 & 2v\\ 0 & 1\end{bmatrix}}$
   \edge node[auto=right] {\tiny{$L_uR_v^2$}};[.${\begin{bmatrix} 1 & 2v\\ u & 1+2uv\end{bmatrix}}$ ]
  \edge node[auto=left] {\tiny{$R_v^3$}}; [.${\begin{bmatrix} 1 & 3v\\ 0 & 1\end{bmatrix}}$ ]
   ] ]]
\end{tikzpicture}
\end{center}
\caption{The first four rows of the $\mathcal{T}^{(u,v)}(I_2)$ tree.}\label{fig:idexample}
\end{figure}

In Figure~\ref{fig:idexample1} in the Appendix we illustrate the $\mathcal{T}^{(1,1)}(I_2)$ tree. By associating a reduced rational number $a/b$ to the vector $\begin{bmatrix}a \\b\end{bmatrix}$ and considering the column vectors of each matrix, it can be seen that $\mathcal{T}^{(1,1)}(I_2)$ is essentially running two Calkin-Wilf trees in parallel that are independent of each other.

This PLFT $(u,v)$-Calkin-Wilf tree perspective helps us understand the monoid generated by $L_u$ and $R_v$, because the left child and the right child of a vertex $M$ are obtained precisely by multiplying $M$ (on the left) by $L_u$ and $R_v$, respectively. Furthermore, the PLFT $(u,v)$-Calkin-Wilf tree organizes the elements in the monoid generated by $L_u$ and $R_v$ by depth $n$ (see Figure~\ref{fig:idexample}). This organization is highly symmetric, a property which is exploited in this paper. A visual inspection of the forms of the entries in terms of $u$ and $v$ in Figure~\ref{fig:idexample} suggests our results. To highlight our key ideas and make them more clear, we provide several examples and a concise summary of our arguments. See Figures~\ref{fig:idexample1}-\ref{fig:idexample4} in the Appendix for an illustration of the first five rows of the $\mathcal{T}^{(1,1)}(I_2)$,
$\mathcal{T}^{(5,2)}(I_2)$, $\mathcal{T}^{(2,5)}(I_2)$, and  $\mathcal{T}^{(5,1)}(I_2)$ trees.

\subsection{Examples}

We begin with a few different examples illustrating various cases. Let $M_{i,j}$ denote the $(i,j)$-entry of $M$.

\begin{itemize}
\item In the case of the $\mathcal{T}^{(1,1)}(I_2)$ tree (see Figure~\ref{fig:idexample1}), the numerical values of the maximal entries and their corresponding positions in depth $n$ are:
\begin{table}[!ht]
\centering
\begin{tabular}{ccccccccc}
1 & $\rightarrow$ &  1 & $\rightarrow$ &  2 & $\rightarrow$ &  3 & $\rightarrow$ &  $\cdots$,\\
$I_{1,1}$  & $\rightarrow$ &  $(L_1)_{2,1}$  & $\rightarrow$ &  $(R_1L_1)_{1,1}$ & $\rightarrow$ & $(L_1R_1L_1)_{2,1}$ & $\rightarrow$ & $\cdots$.
\end{tabular}
\end{table}

\noindent We observe that the first five maximal entries coincide with the first five terms of the Fibonacci sequence. We also observe that a path of maximal entries exhibits an alternating $L_u$ and $R_v$ pattern. A close examination of the $\mathcal{T}^{(1,1)}(I_2)$ tree reveals that the path given above are not unique since
$$I_{1,1} \;\rightarrow\; (R_1)_{1,2} \;\rightarrow\; (L_1R_1)_{2,2}\;\rightarrow\;
(R_1L_1R_1)_{1,2} \;\rightarrow\; \cdots$$
provides another path to obtain the maximal entries. It is also interesting to see that this path is the symmetric image of the previous path.

\item In the case of the $\mathcal{T}^{(5,2)}(I_2)$ tree (see Figure~\ref{fig:idexample2}), the numerical values of the maximal entries and their corresponding positions in depth $n$ are:

\begin{table}[!ht]
\centering
\begin{tabular}{ccccccccc}
1 & $\rightarrow$ &  5 & $\rightarrow$ &  11 & $\rightarrow$ &  60 & $\rightarrow$ & $\cdots$,\\
$I_{1,1}$  & $\rightarrow$ &  $(L_5)_{2,1}$  & $\rightarrow$ &  $(R_2L_5)_{1,1}$ & $\rightarrow$ & $(L_5R_2L_5)_{2,1}$  & $\rightarrow$ &  $\cdots$.
\end{tabular}
\end{table}

\noindent We observe the generalized Fibonacci relations $11=5\cdot v+1$, $60=11\cdot u + 5$ and $131=60\cdot v +11$, and the path exhibits the same alternating $L_u$ and $R_v$ pattern as before. Furthermore, a close examination of $\mathcal{T}^{(5,2)}(I_2)$ reveals that for odd depth, the maximal entry appears to occur uniquely in one position, whereas, for even depth, the same is not true.

\item In the case of the $\mathcal{T}^{(2,5)}(I_2)$ tree (see Figure~\ref{fig:idexample3}), the numerical value of the maximal entries and their corresponding positions at depth $n$ are:

\begin{table}[!ht]
\centering
\begin{tabular}{ccccccccc}
1 & $\rightarrow$ &  5 & $\rightarrow$ &  11 & $\rightarrow$ &  60 & $\rightarrow$ & $\cdots$,\\
$I_{2,2}$  & $\rightarrow$ & $(R_5)_{1,2}$  & $\rightarrow$ & $(L_2R_5)_{2,2}$ & $\rightarrow$ & $(R_5L_2R_5)_{1,2}$  & $\rightarrow$ &  $\cdots$.
\end{tabular}
\end{table}

\noindent Note that the maximal values of the $\mathcal{T}^{(2,5)}(I_2)$ tree are exactly the same as those of $\mathcal{T}^{(5,2)}(I_2)$ tree, and the path of the maximal entries is the symmetric counterpart of the path in $\mathcal{T}^{(5,2)}(I_2)$.

\item In the case of the $\mathcal{T}^{(5,1)}(I_2)$ tree (see Figure~\ref{fig:idexample4}), the numerical values of the maximal entries and their corresponding positions at depth $n$ are:

\begin{table}[!ht]
\centering
\begin{tabular}{ccccccccc}
1 & $\rightarrow$ &  5 & $\rightarrow$ &  10 & $\rightarrow$ &  35 & $\rightarrow$ &  $\cdots$,\\
$I_{2,2}$  & $\rightarrow$ & $(L_5)_{2,1}$ & $\rightarrow$ & $(L_5)^2_{2,1}$ & $\rightarrow$ & $(L_5R_1L_5)_{2,1}$ & $\rightarrow$ &   $\cdots$.
\end{tabular}
\end{table}

\noindent Note that this path does not exhibit the nice alternating $L_u$ and $R_v$ pattern and the generalized Fibonacci relation fails here.  Although the first two terms of the sequence follow the Fibonacci relation, $10 = 5 + 1\cdot u$ and $35 = 10 + 5\cdot u$, the third term does not since $65\neq 35 + 10\cdot u$.
\end{itemize}

\subsection{The maximal entry in $\mathcal{T}^{(u,v)}(I_2;n)$}

We denote by $\mathcal{T}^{(u,v)}(M;n)$ the (finite) set of matrices of depth $n$ in $\mathcal{T}^{(u,v)}(M)$, where $n\ge 0$. We often refer to $\mathcal{T}^{(u,v)}(M;n)$ as the $n^\text{th}$ row of the tree. In our search for the maximal entry among all matrices of $\mathcal{T}^{(u,v)}(M;n)$, not only do we observe Fibonacci-like sequence relations among the maximal entries, we are also able to provide an explicit formula for the maximal entry at each depth.

We define the function $\mu:GL_2(\mathbb{N}_0)\to \mathbb{N}$ by $\mu\left(\begin{bmatrix} a & b\\ c & d\end{bmatrix}\right)=\max\{a,b,c,d\}$, which computes the maximal entry of a matrix $M\in GL_2(\mathbb{N}_0)$. For a finite subset $S$ of $GL_2(\mathbb{N}_0)$, we extend the definition of $\mu$ to $S$ by $\mu(S)=\max\{\mu(M):M\in S\}$. So $\mu(\mathcal{T}^{(u,v)}(M;n))$ is the maximal entry among all matrices of depth $n$ in the $\mathcal{T}^{(u,v)}(M)$ tree.
In our main theorem (Theorem 1) we show that for $n\geq 0$ and fixed $u\geq v \geq 1$:

\begin{align*}
\mu(\mathcal{T}^{(u,v)}(I_2;n)) = &
 \begin{cases}
  ((L_uR_v)^{(n-1)/2}L_u)_{2,1} & \text{ when } n \text{ is odd}, \\
  (L_u(L_uR_v)^{n/2}L_u)_{2,1}  & \text{ when } n \text{ is even and } u>v=1,\\
  ((R_vL_u)^{n/2})_{1,1} & \text{ otherwise}.
\end{cases} \\
\end{align*}
Similarly, by symmetry, for $n\geq 0$ and fixed $v\geq u \geq 1$:
\begin{align*}
\mu(\mathcal{T}^{(u,v)}(I_2;n)) = &
 \begin{cases}
  ((R_vL_u)^{(n-1)/2}R_v)_{1,2} & \text{ when } n \text{ is odd}, \\
  (R_v(R_vL_u)^{n/2}R_v)_{1,2}  & \text{ when } n \text{ is even, and } v>u=1,\\
   ((L_uR_v)^{n/2})_{2,2} & \text{ otherwise}.\\
\end{cases} \\
\end{align*}

Theorem~\ref{main} provides an explicit formula in terms of $u$, $v$ and $n$ for computing $\mu(\mathcal{T}^{(u,v)}(I_2;n))$. We define a recursive sequence $F_n^{(u,v)}$ by $F_0^{(u,v)}=1$, $F_1^{(u,v)}=\max\{u,v\}$, and for $n>1$
\begin{align*}
    F_n^{(u,v)}=\begin{cases}
        \max\{u,v\}\cdot F_{n-1}^{(u,v)}+F_{n-2}^{(u,v)} & \text{ for $n$ odd,}\\
        \min\{u,v\}\cdot F_{n-1}^{(u,v)}+F_{n-2}^{(u,v)} & \text{ for $n$ even.}
        \end{cases}
\end{align*}
Theorem~\ref{main} also shows that when $u,v>1$ or $u=v=1$, $\mu(\mathcal{T}^{(u,v)}(I_2;n))=F_n^{(u,v)}$, that is, the maximal entries coincide with the terms of the above generalized Fibonacci sequence. In particular, $\mu(\mathcal{T}^{(1,1)}(I_2;n))=F_n^{(1,1)}$ gives precisely the $n^{\text{th}}$ Fibonacci number. When $u=v$, $F_n^{(u,u)}$ is a Lucas sequence. This result is analogous to a theorem on the largest values of the Stern sequence by Lucas expanded upon by Paulin~\cite{L,P} and is reminiscent of the work Bates, Bunder, and Tognetti~\cite{BBT1,BBT2} in locating specific rationals within the Calkin-Wilf and Stern-Brocot trees.

Our approach in the proof of Theorem~\ref{main} is to consider the entries of each matrix as polynomials in one variable following the $(u,v)$-Calkin-Wilf tree generation rule (see Figures~\ref{fig:idexample} and~\ref{fig:fun_uv}). The difficulty lies in the determination of the different rates of growth of the polynomial functions in Proposition~\ref{polyentry} as the result of the generation rule, that is, whether the left child (multiplication by $u$) has a larger entry than the right child (multiplication by $v$) or vice versa. Another challenge is the case $u> v=1$, which requires special treatment using a polynomial partial ordering; the difficulty is due to a ``misalignment". Namely, the maximal entry in an odd row does not belong to a matrix that is the child of a matrix containing the maximal entry in the previous (even) row.

\begin{figure}[ht!]
\begin{center}
\begin{tikzpicture}[every tree node/.style={font=\small,anchor=base}, sibling distance=10pt]
\tikzset{level distance=55pt}
\Tree[.${\begin{bmatrix}f_1(uv) & f_2(uv) \\ f_3(uv) & f_4(uv)\end{bmatrix}}$ [.${\begin{bmatrix}f_1(uv) & f_2(uv) \\ uf_1(uv)+f_3(uv) & uf_2(uv)+f_4(uv)\end{bmatrix}}$
  ] [.${\begin{bmatrix}vf_3(uv)+f_1(uv) & vf_4(uv)+f_2(uv) \\ f_3(uv) & f_4(uv)\end{bmatrix}}$
 ]]
\end{tikzpicture}
\end{center}
\caption{Entries in the tree in terms of $u$ and $v$.}\label{fig:fun_uv}
\end{figure}

\subsection{Certain symmetric properties of PLFT $(u,v)$-Calkin-Wilf trees}

The symmetric features of $\mathcal{T}^{(u,v)}(I_2)$ have greatly aided our proof by simplifying the computations and reducing the number of cases that need to be considered. Let $n\geq 1$ and $i\in\{1,\ldots,2^n\}$.
We denote by $c_{I_2}^{(u,v)}(n,i)$ the $i^{\text{th}}$ element from left to right and $c_{I_2}^{(u,v)}(n,2^n+1-i)$ the $i^{\text{th}}$ element from right to left of depth $n$  in $\mathcal{T}^{(u,v)}(I_2)$. The elements $c_{I_2}^{(u,v)}(n,i)$ and $c_{I_2}^{(u',v')}(n,2^n+1-i)$ are said to be in \emph{symmetric positions}\footnote{Note that pairs of matrices in symmetric positions do not have to be vertices of the same tree.}.
 Let $\mathcal{L}^{(u,v)}=\mathcal{T}^{(u,v)}(L_u)$ and $\mathcal{R}^{(u,v)}=\mathcal{T}^{(u,v)}(R_v)$ represent the ``left side" and ``right side" of $\mathcal{T}^{(u,v)}(I_2)$.

The following properties are instrumental in our proof:

\begin{itemize}
\item Note that if $1\leq i\leq 2^{n-1}$, then $c_{I_2}^{(u,v)}(n,i)$ is a vertex in $\mathcal{L}^{(u,v)}$ and \\$c_{I_2}^{(u,v)}(n,2^n+1-i)$ is a vertex in $\mathcal{R}^{(u,v)}$. It is easy to see that if $$c_{I_2}^{(u,v)}(n,i)=\cdots R^{\alpha_4}_vL^{\alpha_3}_uR^{\alpha_2}_vL^{\alpha_1}_u,$$ then $$c_{I_2}^{(u',v')}(n,2^n+1-i)=\cdots L^{\alpha_4}_{u'}R^{\alpha_3}_{v'}L^{\alpha_2}_{u'}R^{\alpha_1}_{v'}$$ when $1\leq i\leq 2^{n-1}$.
\item If $c_{I_2}^{(u,v)}(n,i)= \begin{bmatrix} a & b\\ c & d \end{bmatrix}$, then
$c_{I_2}^{(v,u)}(n,2^n+1-i)= \begin{bmatrix} d & c\\ b & a \end{bmatrix}$ (Proposition~\ref{uvvu}). Hence, we may assume $u\geq v $ since, from the above symmetry, it follows that $\mu(\mathcal{T}^{(u,v)}(I_2;n))=\mu(\mathcal{T}^{(v,u)}(I_2;n))$. That is, the maximal entry among matrices of depth $n$ of $\mathcal{T}^{(u,v)}(I_2)$ is the same as that of $\mathcal{T}^{(v,u)}(I_2)$.

\item If $c_{I_2}^{(u,v)}(n,i)=\begin{bmatrix}a & b \\ c & d\end{bmatrix}$, then $c_{I_2}^{(u,v)}(n,2^n+1-i)=\begin{bmatrix}d & \frac{cv}{u} \\ \frac{bu}{v} & a\end{bmatrix}$ (Proposition~\ref{polyentry} and~\cite[Theorem 1]{HMST2}).  This symmetric property together with Proposition~\ref{leftdom} and the assumption that $u\geq v$ leads to the conclusion that the maximal entry occurs in a matrix on the left side of the tree.

\item Let $M=\begin{bmatrix}a & b \\ c & d\end{bmatrix}$ be a matrix in $\mathcal{T}^{(u,v)}(I_2)$. Then Lemma~\ref{domcol} shows that the maximal entry on the left side of the tree must appear in the first column of the matrix, that is, $\mu(M)=\max\{a,c\}$. Likewise the maximal entry on the right side of the tree must appear in the second column of the matrix, that is, $\mu(M)=\max\{b,d\}$. Given the assumption $u\geq v$, we focus our work on the first columns of the left side of the tree.

\end{itemize}

Proposition~\ref{ineq} provides an expression for the maximal entry for odd rows. Let $(L_uR_v)^nL_u=\begin{bmatrix} A_n & * \\C_n & *
    \end{bmatrix} \in \mathcal{T}^{(u,v)}(I_2;2n+1)$. If
$M=\begin{bmatrix} a & * \\ c & *
    \end{bmatrix}\in \mathcal{T}^{(u,v)}(I_2;2n+1)$, then by induction on $n$, we are able to show that $a,c\leq C_n$.  Moreover, we  provide an explicit formula for $C_n$ by solving a discrete dynamical system (Lemma~\ref{system} and  Proposition~\ref{alternating}).

For the even rows, we use a convenient argument that reduces the problem to the odd row case.  Namely, we consider the maximal entry of $\mathcal{T}^{(u,v)}(R_v;2n+1)$. The maximal entry has the form $\mu(\mathcal{T}^{(u,v)}(R_v;2n+1)) = \mu((R_vL_u)^{n+1})$ by Proposition~\ref{leftdom}. We then show that $\mu(\mathcal{T}^{(u,v)}(I_2;2n+2))=\mu(\mathcal{T}^{(u,v)}(R_v;2n+1))$.

\subsection{An interesting application in cryptography}

Another goal of this paper is to address a question raised in~\cite{B,BSV} related to Cayley hash functions. We provide an answer as  an application of Theorem~\ref{main}.

A hash function is a function that accepts data of arbitrary size as an input and produces an output of a fixed size. For example, the function $f:\mathbb{N}\to[0,m)$ given by $f(n)=n\pmod{m}$ always outputs a nonnegative integer that is no larger than $m-1$, regardless of the size of the input. This can be a useful tool in storing data (such as online passwords). This leads one to demand that a desirable hash function satisfy some basic requirements (as seen in~\cite{BSV}):

\begin{enumerate}
\item It should be computationally difficult to determine an input that hashes to a given output.
\item It should be computationally difficult to determine a second input that hashes to the same output as another given input.
\item It should be computationally difficult to determine two inputs that hash to the same output (referred to as collision resistance).
\end{enumerate}

In~\cite{BSV}, Bromberg, Shpilrain, and Vdovina define a Cayley hash function, which we refer to as the BSV hash\footnote{The BSV hash is a generalization of a hash function defined by Z\'emor~\cite{Z}.}, for binary strings in the following way.
Let $p$ be a large prime. For fixed integers $u,v\geq 1$ and a binary string $w=a_0a_1\cdots a_n$ where $a_i\in\{0,1\}$ for $i=0,\dots, n$, let $M=\prod_{i=0}^n f(a_i)$ where $f(0)=L_u$ and $f(1)=R_v$. (For the empty string $\lambda$, define $f(\lambda)=I_2$.) The hashed output, a matrix in $SL_2(\mathbb{F}_p)$, is obtained by reducing the entries of $M$ modulo $p$. For example, when $u=2$, $v=3$ and $p=5$, the hashed output of the string $01100$ is given by $\begin{bmatrix}0 & 1\\4 & 3\end{bmatrix}$.

There is a natural one-to-one correspondence between matrices in $\mathcal{T}^{(u,v)}(M;n)$ and bit strings of length $n$.  As such, if
one can show that $\mu(\mathcal{T}^{(u,v)}(M;n))$  is bounded above by some monotonically increasing function $f_{(u,v)}(n)$ for all $n$, then the conclusion is that, in the BSV hash,  all bit strings of length at most $n_0:=n_0(u,v)$ have distinct hashed values (i.e., there are no collisions for pairs of ``short" strings) where $n_0$ is the largest integer such that $f_{(u,v)}(n_0)<p$. In~\cite{BSV}, this is precisely what is done in the cases $u=v=2$ and $u=v=3$.

\medskip

\begin{proposition}[Bromberg, Shpilrain, and Vdovina~\cite{BSV}]\label{bsvext}
In the BSV hash there are no collisions between outputs of length less than
\[
\begin{cases}
\log_{\sqrt{3+\sqrt 8}}{p} & \text{when } u=v=2\text{ and}\\
\log_{\sqrt{\frac{11+\sqrt{117}}{2}}}{p} & \text{when } u=v=3.
\end{cases}
\]
\end{proposition}

\medskip

Corollary~\ref{hash}, a consequence of Theorem~\ref{main}, immediately allows us to extend this result to any $u,v \geq 1$. In particular, in the case where $u=2$ and $v=1$, we get that in the BSV hash there are no collisions between outputs of length less than $\log_{\sqrt{2+\sqrt 3}}{p}$ where $\sqrt{2+\sqrt 3}\approx 1.9,$ answering a question posed in~\cite[Problem 3]{BSV}.

%In~\cite[Problem 3]{BSV}, the authors raise the question of extending Proposition~\ref{bsvext} to the case $u=1, v=2$. Corollary~\ref{hash} allows us to answer this question. In particular, in this case, we get that in the BSV hash there are no collisions between outputs of length less than $\log_{\sqrt{2+\sqrt 3}}{p}$ where $\sqrt{2+\sqrt 3}\approx 1.9.$

The remainder of this paper is organized in the following way. Section 2 contains our main result. Section 3 is devoted to giving a thorough proof of the main result. The proof involves a careful analysis of various cases using a series of different techniques.

%%%%%%%%%%%%%%%%%%%%%%%%%%%%%%%%%%%%%%%%%%%%%%%%%%

\section{Main results}\label{mt}

We now state our main results. First, we provide a monotonic closed formula for the maximal entry of any matrix of depth $n$ for any $u,v\geq 1$, denoted by $\mu(\mathcal{T}^{(u,v)}(I_2;n))$, depending on the parity of $n$. We also determine a matrix and position entry that contain the maximal value for every depth $n$. As a consequence, we obtain an upper bound on the binary string length that guarantees collision resistance on the BSV Cayley hash function for any $u,v\geq 1$.

\begin{theorem}\label{main}
For $n\geq 0$ and positive integers $u$ and $v$, let $s_{u,v}=\min\{u,v\}$, $t_{u,v}=\max\{u,v\}$, $p_{u,v}^\pm=\pm s_{u,v}\sqrt{t_{u,v}}+\sqrt{s_{u,v}(4+uv)}$ and $q_{u,v}^\pm=2+uv\pm\sqrt{uv(4+uv)}$. Then
\begin{align}
\mu(\mathcal{T}^{(u,v)}(I_2;2n+1)) &= \frac{\sqrt{t_{u,v}}\left((q_{u,v}^+)^{n+1}-(q_{u,v}^-)^{n+1}\right)}{2^{n+1}\sqrt{s_{u,v}(4+uv)}}\label{oddmax}
\end{align}
and
\begin{eqnarray}
\nonumber&&\mu(\mathcal{T}^{(u,v)}(I_2;2n+2))\\
&=& \begin{cases} \frac{p_{u,v}^+(q_{u,v}^+)^{n+1}+p_{u,v}^-(q_{u,v}^-)^{n+1}}{2^{n+2}\sqrt{s_{u,v}(4+uv)}} & \text{if $s_{u,v}>1$},\\
\frac{\sqrt{t_{u,v}}\left((\sqrt{t_{u,v}}p_{u,v}^-+2)(q_{u,v}^+)^{n+1}+(\sqrt{t_{u,v}}p_{u,v}^+-2)(q_{u,v}^-)^{n+1}\right)}{2^{n+2}\sqrt{(4+uv)}} & \text{otherwise}.\label{evenmax}
\end{cases}
\end{eqnarray}
Furthermore, the value given by~\eqref{oddmax} is attained by the $(2,1)$ entry of the matrix $(L_uR_v)^nL_u$ when $u\geq v$ and by the $(1,2)$ entry of the matrix $(R_vL_u)^nR_v$ when $v\geq u$. Similarly, the value given by~\eqref{evenmax} is attained by the $(1,1)$ entry of the matrix $(R_vL_u)^n$ when $u\geq v>1$, by the $(2,1)$ entry of the matrix $L_u(L_uR_v)^{n-1}L_u$ when $u\geq v=1$, by the $(2,2)$ entry of the matrix $(L_uR_v)^n$ when $v\geq u>1$, and by the $(1,2)$ entry of the matrix $R_v(R_vL_u)^{n-1}R_v$ when $v\geq u=1$.
\end{theorem}

For concrete examples of Theorem~\ref{main}, see Tables~\ref{tab:odduvtable} and~\ref{tab:evenuvtable}.

\begin{table}[ht!]
\centering
\begin{tabular}{|c||*{2}{c|}}\hline
\backslashbox{$u$ \kern-2em}{$v$}
&1&2\\\hline\hline
1 &$\frac{(3+\sqrt{5})^{n+1}-(3-\sqrt{5})^{n+1}}{2^{n+1}\sqrt{5}}$ & $\frac{(2+\sqrt{3})^{n+1}-(2-\sqrt{3})^{n+1}}{\sqrt{3}}$\\\hline
2 &$\frac{(2+\sqrt{3})^{n+1}-(2-\sqrt{3})^{n+1}}{\sqrt{3}}$&$\frac{(3+2\sqrt{2})^{n+1}-(3-2\sqrt{2})^{n+1}}{2\sqrt{2}}$\\\hline
3 &$\frac{3((5+\sqrt{21})^{n+1}-(5-\sqrt{21})^{n+1})}{2^{n+1}\sqrt{21}}$&$\frac{3((4+\sqrt{15})^{n+1}-(4-\sqrt{15})^{n+1})}{2\sqrt{15}}$\\\hline
\end{tabular}
\caption{The value of $\mu(\mathcal{T}^{(u,v)}(I_2;2n+1))$ for various choices of $u$ and $v$.}\label{tab:odduvtable}
\end{table}

\begin{table}[ht!]
\centering
\begin{tabular}{|c||*{1}{c|}}\hline
\backslashbox{$u$ \kern-2em}{$v$}
&1\\\hline\hline
1 &$\frac{(\sqrt{5}+1)(3+\sqrt{5})^{n+1}+(\sqrt{5}-1)(3-\sqrt{5})^{n+1}}{2^{n+2}\sqrt{5}}$\\\hline
2 &$(2+\sqrt{3})^{n+1}+(2-\sqrt{3})^{n+1}$\\\hline
3 & $\frac{3((\sqrt{21}-1)(5+\sqrt{21})^{n+1}+(\sqrt{21}+1)(5-\sqrt{21})^{n+1})}{2^{n+2}\sqrt{21}}$\\\hline
\end{tabular}
\caption{The value of $\mu(\mathcal{T}^{(u,v)}(I_2;2n+2))$ for various choices\protect\footnotemark of $u$ and $v$.}\label{tab:evenuvtable}
\end{table}

\footnotetext{We restrict the number of entries in this table due to space considerations.}

The discussion in the introduction, together with Theorem~\ref{main}, immediately gives the following result.

\begin{corollary}\label{hash}
Let $u,v\geq 1$ and $n_0:=n_0(u,v)$ be the largest integer such that $\mu(\mathcal{T}^{(u,v)}(I_2;n_0))< p$. Then there are no collisions between distinct bit strings of length $\leq n_0$ in the BSV hash.
\end{corollary}

When $u=v\in\{2,3\}$, our results in Theorem~\ref{main} and Corollary~\ref{hash} coincide with the ones obtained by Bromberg, Shpilrain, and Vdovina in~\cite{BSV}.

%%%%%%%%%%%%%%%%%%%%%%%%%%%%%%%%%%%%%%%%%%%%%%%%%%

\section{Proof of Theorem~\ref{main}}

For the remainder of the paper, since we are concentrating on a proof of Theorem~\ref{main}, which involves the  $\mathcal{T}^{(u,v)}(I_2)$ tree, we will focus our attention only on matrices in $SL_2(\mathbb{N}_0)$.

In Theorem~\ref{main}, the claim is that, when $u\geq v$, $(L_uR_v)^nL_u$ has the maximal entry among all other matrices in $\mathcal{T}^{(u,v)}(I_2;2n+1)$. We first show that the left column entries of matrices of this form can be easily computed using a discrete dynamical system.

\begin{lemma}\label{system}
Let $u,v\in\mathbb{N}$ and $a,c\in\mathbb{N}_0$ (not both zero). Define $\alpha_n:=\alpha_n^{(u,v)}(a,c)$ and $\gamma_n:=\gamma_n^{(u,v)}(a,c)$ recursively by
\begin{align*}
\alpha_n &=\begin{cases}a & \text{ for $n=0$,}\\
\alpha_{n-1}+v\gamma_{n-1} & \text{ otherwise}
\end{cases}
\end{align*}
and
\begin{align*}
\gamma_n &=\begin{cases}ua+c & \text{ for $n=0$,}\\
u\alpha_{n-1}+(1+uv)\gamma_{n-1} & \text{ otherwise.}
\end{cases}
\end{align*}
Then $\gamma_n\geq \alpha_n$,
\begin{align*}
\gamma_n &= \frac{(cp_{u,v}^++aq_{u,v}^+\sqrt{u})(q_{u,v}^+)^n+(cp_{u,v}^--aq_{u,v}^-\sqrt{u})(q_{u,v}^-)^n}{2^{n+1}\sqrt{v(4+uv)}},\\
\end{align*}
and
\begin{align*}
\alpha_n &= \frac{(cp_{u,v}^++aq_{u,v}^+\sqrt{u})(q_{u,v}^+)^np_{u,v}^--(cp_{u,v}^--aq_{u,v}^-\sqrt{u})(q_{u,v}^-)^np_{u,v}^+}{2^{n+2}\sqrt{uv(4+uv)}}\\
\end{align*}
where $p_{u,v}^\pm=\pm v\sqrt{u}+\sqrt{v(4+uv)}$ and $q_{u,v}^\pm=2+uv\pm\sqrt{uv(4+uv)}$.
\end{lemma}

\begin{proof}
It is clear that $\gamma_0\geq \alpha_0$.  The fact that $\gamma_n\geq \alpha_n$ for $n\geq 1$ follows from noticing that $\gamma_n=u\alpha_n+\gamma_{n-1}$.

As a matrix equation, we have that, for $n\geq 1$,
\begin{align*}
\begin{bmatrix}\alpha_n\\ \gamma_n\end{bmatrix} &= \begin{bmatrix}1 & v\\ u & 1+uv\end{bmatrix} \begin{bmatrix}\alpha_{n-1}\\ \gamma_{n-1}\end{bmatrix}.
\end{align*}
The eigenvalues of the matrix $\begin{bmatrix}1 & v\\ u & 1+uv\end{bmatrix}$ are
$$\lambda_1=\frac{1}{2}\left(2+uv+\sqrt{uv(4+uv)}\right)\quad \text{ and }\quad \lambda_2=\frac{1}{2}\left(2+uv-\sqrt{uv(4+uv)}\right)$$ with associated eigenvectors $\vec{v}_1=\begin{bmatrix}\frac{\sqrt{v(4+uv)}-v\sqrt{u}}{2\sqrt{u}}\\1\end{bmatrix}$ and $\vec{v}_2=\begin{bmatrix}\frac{-\sqrt{v(4+uv)}-v\sqrt{u}}{2\sqrt{u}}\\1\end{bmatrix}$, respectively. Solving the vector equation \begin{align*}
\begin{bmatrix}\alpha_0\\ \gamma_0\end{bmatrix} = c_1\vec{v}_1+c_2\vec{v}_2
\end{align*}
gives that
\begin{align*}
c_1 &= \frac{c(v\sqrt{u}+\sqrt{v(4+uv)})+a\sqrt{u}(2+uv+\sqrt{uv(4+uv)})}{2\sqrt{v(4+uv)}}
\end{align*}
and
\begin{align*}
c_2 &=\frac{c(-v\sqrt{u}+\sqrt{v(4+uv)})-a\sqrt{u}(2+uv-\sqrt{uv(4+uv)})}{2\sqrt{v(4+uv)}}.
\end{align*}
It follows that
\begin{align*}
\begin{bmatrix}\alpha_n\\ \gamma_n\end{bmatrix} &= \begin{bmatrix}1 & v\\ u & 1+uv\end{bmatrix}^n \begin{bmatrix}\alpha_0\\ \gamma_0\end{bmatrix}\\
&= \begin{bmatrix}1 & v\\ u & 1+uv\end{bmatrix}^n(c_1\vec{v}_1+c_2\vec{v}_2)\\
&= c_1\lambda_1^n\vec{v}_1+c_2\lambda_2^n\vec{v}_2.
\end{align*}
So
$
\gamma_n = c_1\lambda_1^n+c_2\lambda_2^n,
$
which gives the desired result after the appropriate substitutions.
\end{proof}

\begin{proposition}\label{alternating}
Suppose that $M\in SL_2(\mathbb{N}_0)$ is given by $M=\begin{bmatrix} a & b\\ c & d\end{bmatrix}$.
For any $n\geq 0$, let
\begin{align*}
(L_uR_v)^nL_uM &= \begin{bmatrix} A_n & \ast\\ C_n & \ast\end{bmatrix}.
\end{align*}
Then $A_n=\alpha_n$ and $C_n=\gamma_n$ where $\alpha_n$ and $\gamma_n$ are as defined in Lemma~\ref{system}.
\end{proposition}

\begin{proof}
The result follows by noting the relationship between the left columns of $(L_uR_v)^nL_uM$ and $(L_uR_v)^{n+1}L_uM$.
\end{proof}

Note that a result similar to Proposition~\ref{alternating} could easily be found for the right column of $(L_uR_v)^nL_uM$. However, as we will see later on, this is not necessary. The symmetries associated with PLFT $(u,v)$-Calkin-Wilf trees will allow us to reduce the number of cases to be analyzed.

With Proposition~\ref{alternating} applied to $I_2$, we can compute the entries in the left column of a specific family of matrices, namely matrices of the form $(L_uR_v)^nL_u$. The next step will be to show that the left column entries of any matrix of depth $2n+1$ are no larger than $C_n$.

\begin{definition}\label{uplow}
Let $M\in SL_2(\mathbb{N}_0)$ be given by $M=\begin{bmatrix} a & b\\ c & d\end{bmatrix}$. We say that $M$ is {\emph{$u$-lower dominant ($u$-LD)}} if $c\geq ua$ and $d\geq ub$ and we say that $M$ is {\emph{$v$-upper dominant ($v$-UD)}} if $a\geq vc$ and $b\geq vd$.
\end{definition}

We get the following immediate consequences of the definitions of $u$-LD and $v$-UD.

\begin{lemma}\label{uld1}
A matrix in $SL_2(\mathbb{N}_0)$ is $u$-LD ($v$-UD) if and only if it is of the form $L_uM$ ($R_vM$) for some $M\in SL_2(\mathbb{N}_0)$.
\end{lemma}

\begin{proof}
Let $M=\begin{bmatrix} a & b\\ c & d\end{bmatrix}$. We have that $L_uM=\begin{bmatrix} a & b\\ ua+c & ub+d\end{bmatrix}$. Clearly we have that $ua+c\geq ua$ and $ub+d\geq ub$, which give the needed inequalities. The remaining part of the proof is similar.
\end{proof}

\begin{lemma}\label{allUDLD}
Suppose that $M\in SL_2(\mathbb{N}_0)$ and let $M'\in\mathcal{T}^{(u,v)}(M;n)$ for some $n>0$. Then $M'$ is either $u$-LD or $v$-UD.
\end{lemma}

\begin{proof}
If $M'\in\mathcal{T}^{(u,v)}(M;n)$, then either $M'=L_uM''$ or $M'=R_vM''$ for some $M''\in\mathcal{T}^{(u,v)}(M;n-1)$. By Lemma~\ref{uld1}, the result follows.
\end{proof}

At this time we consider two separate cases. In the first case we assume that $u\geq v\geq 2$ and in the second that $u\geq v=1$. The proof of the first case is fairly straightforward and mimics many of the parts in the Bromberg, Shpilrain, and Vdovina proof~\cite{BSV}. The second case is more involved and requires a somewhat different approach.

\begin{proposition}\label{ineq}
Let $u\geq v\geq 2$. Suppose that $M,M'\in SL_2(\mathbb{N}_0)$, given by $M=\begin{bmatrix} a & \ast\\ c & \ast\end{bmatrix}$ and $M'=\begin{bmatrix} a' & \ast\\ c' & \ast\end{bmatrix}$, are such $M'\in \mathcal{T}^{(u,v)}(M,2n+1)$ and $a\geq c$. Then $\min\{a',c'\}\leq A_n$ and $\max\{a',c'\}\leq C_n$, where $A_n$ and $C_n$ are as defined in Proposition~\ref{alternating}. Furthermore, when $M'$ is $v$-UD, $a'\leq\frac{v}{u}C_n$.
\end{proposition}

\begin{proof}
For $n=0$, notice that $L_uM=\begin{bmatrix} a & \ast\\ ua+c & \ast\end{bmatrix}$ and $R_vM=\begin{bmatrix} a+vc & \ast\\ c & \ast\end{bmatrix}$ are the only two matrices in $\mathcal{T}^{(u,v)}(M;1)$. Since $(v-1)c\leq (u-1)a$, the result holds in this case.

Suppose that the statement is true for all matrices of depth $2k+1$, for some $k\geq 0$. Let $M'\in\mathcal{T}^{(u,v)}(M,2k+3)$. Then $M'\in\mathcal{T}^{(u,v)}(M'',2)$ for some $M''\in\mathcal{T}^{(u,v)}(M,2k+1)$ given by $M''=\begin{bmatrix} a'' & \ast\\ c'' & \ast\end{bmatrix}$. It must be the case that $$M'\in\{L_u^2M'', L_uR_vM'', R_vL_uM'', R_v^2M''\}.$$ In particular,
\begin{align*}
M'=\begin{cases}
\begin{bmatrix} a'' & \ast\\ 2ua''+c'' & \ast\end{bmatrix} & \text{ if $M'=L_u^2M''$,}\\[3ex]
\begin{bmatrix} a''+vc'' & \ast\\ ua''+(1+uv)c'' & \ast\end{bmatrix} & \text{ if $M'=L_uR_vM''$,}\\[3ex]
\begin{bmatrix} (1+uv)a''+vc'' & \ast\\ ua''+c'' & \ast\end{bmatrix} & \text{ if $M'=R_vL_uM''$,}\\[3ex]
\begin{bmatrix} a''+2vc'' & \ast\\ c'' & \ast\end{bmatrix} & \text{ if $M'=R_v^2M''$.}
\end{cases}
\end{align*}
If $M''$ is $u$-LD, then $ua''\leq c''$, so
\begin{align*}
2ua''+c'' &= ua''+ua''+c''\\
&\leq ua''+2c''\\
&\leq ua''+(1+uv)c''.
\end{align*}
We have that
\begin{align*}
(1+uv)a''+vc'' &= a''+uva''+vc''\\
&\leq a''+2vc''.
\end{align*}
Finally, it follows that $2v\leq 1+uv$ since $u\geq 2$, so $a''+2vc'' \leq ua''+(1+uv)c''$.
These inequalities show that $\max\{a',c'\} \leq ua''+(1+uv)c''.$

It is easy to see that $\min\{a',c'\}$ is at most either $a''+vc''$ or $ua''+c''$.

Since, by assumption, $a''\leq A_k$ and $c''\leq C_k$, it follows that
\begin{align*}
ua''+(1+uv)c'' &\leq uA_k+(1+uv)C_k\\
&= C_{k+1},
\end{align*}
\begin{align*}
a''+vc'' &\leq A_k+vC_k\\
&= A_{k+1},
\end{align*}
and
\begin{align*}
ua''+c'' &\leq 2c''\\
&\leq A_k+vC_k\\
&= A_{k+1},
\end{align*}
since $v\geq 2$, as desired.

If $M''$ is $v$-UD, then one can show that $\max\{a',c'\}\leq (1+uv)a''+vc''$ using a very similar set of arguments as above. The needed inequalities follow from the fact that $vc''\leq a''$ and $u\geq v\geq 2$ in this case. This completes the proof that $\min\{a',c'\}\leq A_n$ and $\max\{a',c'\}\leq C_n$.

For the remainder of the proof, we assume that $M'$ is $v$-UD. If $M''$ is $u$-LD, then
\begin{align*}
    u((1+uv)a''+vc'') &= u(a''+uva''+vc'')\\
    &\leq u(a''+2vc'')\\
    &\leq ua''+uv^2c''\\
    &\leq v(uA_k+(1+uv)C_k)\\
    &= vC_{k+1}.
\end{align*}
Similarly, $M''$ is $v$-UD, then
\begin{align*}
    u(a''+2vc'') &\leq u(2a''+vc'')\\
    &\leq u((1+uv)a''+vc'')\\
    &\leq u((1+uv)\frac{v}{u}C_k+vA_k)\\
    &= v(uA_k+(1+uv)C_k)\\
    &= vC_{k+1}
\end{align*}
where the third inequality follows by the induction hypothesis. Having exhausted all possibilities, we obtain that $a'\leq\frac{v}{u}C_n$.

\end{proof}

A careful reading of the proof above will show that the assumption that $u\geq v\geq 2$ was needed to ensure that the inequalities $2v\leq 1+uv$ and $2u\leq 1+uv$ both hold true. If $v=1$, then the second inequality does not hold in general. We begin our alternate approach with a critical definition.

\begin{definition}\label{def} Let $f(x)=\sum_{i=0}^na_ix^i$ and $g(x)=\sum_{i=0}^mb_ix^i$ be polynomials over $\mathbb{N}_0$. If $\sum_{k\geq N}a_k\geq \sum_{k\geq N}b_k$ for every nonnegative integer $N,$ then we say that $f(x)\succcurlyeq g(x)$. Here we assume that $a_i=0$ for $i>n$ and $b_j=0$ for $j>m$.
\end{definition}

Note some properties of the above definition.
\begin{enumerate}
\item The relation
is a partial order.
\item If  $f(x)\succcurlyeq g(x)$, then $\deg(f)\ge \deg(g)$.
\item If $f_1(x)\succcurlyeq g_1(x)$ and $f_2(x)\succcurlyeq g_2(x)$, then $f_1(x)+f_2(x)\succcurlyeq g_1(x)+g_2(x)$.
\item If $f(x)\succcurlyeq g(x)$ and $g(x)\succcurlyeq h(x)$, then $f(x)\succcurlyeq h(x)$.
\item If $f(x)=g(x)+h(x)$ for some polynomial $h(x)$ over $\mathbb{N}_0$, then $f(x)\succcurlyeq g(x)$.
\item We have that $x^if(x)\succcurlyeq x^jf(x)$ for $i\ge j \ge 0$. (This is due to a simple shift in the coefficients of the polynomial $f(x)$.)
\item If $a_i\ge b_i$ for each $i$ then $\sum_{i=0}^na_ix^i \succcurlyeq \sum_{i=0}^mb_ix^i$.
\end{enumerate}

The importance of Definition~\ref{def} appears in the following lemma. It is a straightforward property that can be used to determine if one polynomial is greater than or equal to another when evaluated over positive integers.

\begin{lemma}\label{domineq}
If $f(x)\succcurlyeq g(x)$, then $f(r)\geq g(r)$ for every positive integer $r$.
\end{lemma}

\begin{proof}
Suppose $f(x)=\sum_{i=0}^na_ix^i$ and $g(x)=\sum_{i=0}^mb_ix^i$ where $a_n, b_m\ne 0$. By hypothesis, we must have $n\geq m$.

Suppose that $b_{m_0}$ is such that $b_{m_0}>a_{m_0}$ and $b_i\leq a_i$ for all $i>m_0$. Let $\epsilon_i=a_i-b_i$ for $i>m_0$ and define a new polynomial $f_{m_0}(x)=\sum_{i=0}^nc_ix^i$ by
\begin{align*}
f_{m_0}(x) & =\sum_{i=m_0+1}^{n}(a_i-\epsilon_i)x^i+\left(a_{m_0}+\sum_{i=m_0+1}^{n}\epsilon_i\right)x^{m_0}+\sum_{i=0}^{m_0}a_ix^i.
\end{align*}
It follows that $f_{m_0}(x)\succcurlyeq g(x)$ and that $b_i\leq c_i$ for all $i\geq m_0$. Furthermore,
\begin{align*}
f(r) &= \sum_{i=0}^na_ir^i\\
&= \sum_{i=m_0+1}^{n}(a_i-\epsilon_i+\epsilon_i)r^i+a_{m_0}r^{m_0}+\sum_{i=0}^{m_0}a_ir^i\\
&\geq \sum_{i=m_0+1}^{n}(a_i-\epsilon_i)r^i+\left(a_{m_0}+\sum_{i=m_0+1}^{n}\epsilon_i\right)r^{m_0}+\sum_{i=0}^{m_0}a_ir^i\\
&= f_{m_0}(r).
\end{align*}
Iterating this procedure will generate a finite list of polynomials $f_{m_0}(x)$, $f_{m_1}(x),\dots,$ $f_{m_k}(x)$ with $f(r)\geq f_{m_0}(r)\geq \cdots\geq f_{m_k}(r)$ and $f_{m_k}(x)=\sum_{i=0}^nd_ix^i$ such that $d_i\geq b_i$ for all $1\leq i\leq n$. Clearly $f_{m_k}(r)\geq g(r)$, which gives the desired result.
\end{proof}

Note that the converse of Lemma~\ref{domineq} is not true. If $f(x)=x^3+1$ and $g(x)=x^2+x$, then $f(r)\geq g(r)$ for every positive integer $r$, but it is \textbf{not} true that $f(x)\succcurlyeq g(x)$.

In order to apply Lemma~\ref{domineq} to our current case, we first show that the left column entries of matrices appearing in $\mathcal{T}^{(u,1)}(I_2)$ can all be expressed as polynomials evaluated at $u$. We also explicitly compute such polynomials for certain families of matrices, namely matrices of the form $(L_uR_1)^nL_u$ and $(R_1L_u)^nL_u$.

\begin{lemma}\label{polyrep}
Let $M'\in\mathcal{T}^{(u,1)}(M;n)$ be given by $M'=\begin{bmatrix} a' & \ast\\ c' & \ast\end{bmatrix}$. Then $a'=f(u)$ and $c'=g(u)$ where $f(x)$ and $g(x)$ are polynomials over $\mathbb{N}_0$ with $f(0)=1$ and $g(0)=0$.
\end{lemma}

\begin{proof}
Clearly the statement is true for $n=0$.

Suppose that the statement holds for all matrices of depth $k$ for some $k\geq 0$. Let $M'\in\mathcal{T}^{(u,1)}(M;k+1)$. It follows that $M'=L_uM''$ or $M'=R_1M''$ for some $M''\in\mathcal{T}^{(u,1)}(M;k)$. By assumption, $M''=\begin{bmatrix} f(u) & \ast\\ g(u) & \ast\end{bmatrix}$ for some polynomials $f(x)$ and $g(x)$ over $\mathbb{N}_0$. It follows that $L_uM''=\begin{bmatrix} f(u) & \ast\\ uf(u)+g(u) & \ast\end{bmatrix}$ and $R_1M''=\begin{bmatrix} f(u)+g(u) & \ast\\ g(u) & \ast\end{bmatrix}$. In either case, it is obvious that the statement holds for $M'$, which gives the result by induction.
\end{proof}

Note that the polynomials in Lemma~\ref{polyrep} depend on $M$, but not on the value of $u$.

We will make extensive use of the following result based on Pascal's rule that $\binom{n-1}{k-1}+\binom{n-1}{k}=\binom{n}{k}$ for $1\leq k\leq n$.

\begin{lemma}\label{aux}
We have that \[\displaystyle\sum_{i=0}^{a-1}\binom{b-i}{i} x^{a-i} + \sum_{i=0}^{a}\binom{b+1-i}{i} x^{a+1-i}
=\sum_{i=0}^{a}\binom{b+2-i}{i} x^{a+1-i}.\]
\end{lemma}
\begin{proof}
\begin{eqnarray*}
&&\displaystyle\sum_{i=0}^{a-1}\binom{b-i}{i} x^{a-i} + \sum_{i=0}^{a}\binom{b+1-i}{i} x^{a+1-i} \\
&=& \displaystyle\sum_{i=1}^{a}\binom{b+1-i}{i-1} x^{a+1-i} + \sum_{i=0}^{a}\binom{b+1-i}{i} x^{a+1-i} \\
&=& \displaystyle\sum_{i=1}^{a}\left[\binom{b+1-i}{i-1}  + \binom{b+1-i}{i}  \right]x^{a+1-i}+x^{a+1} \\
&=& \sum_{i=0}^{a}\binom{b+2-i}{i} x^{a+1-i}
\end{eqnarray*}
\end{proof}

Lemma~\ref{aux} contains an identity involving binomial coefficients that we apply multiple times with different parameters in Lemma~\ref{FnGn}, Lemma~\ref{HnIn}, Proposition~\ref{sim}, and Proposition~\ref{sim2}. We elected to carefully document our use of the lemma and the partial ordering of polynomials for clarity. The reader may choose to skip to Proposition~\ref{v1version} after reading Lemma~\ref{FnGn}.

\begin{lemma}\label{FnGn}
For any $n\geq 0$, let $F_n(x)$ and $G_n(x)$ be the polynomials over $\mathbb{N}_0$ such that $(L_uR_1)^nL_u=\begin{bmatrix} F_n(u) & \ast\\ G_n(u) & \ast\end{bmatrix}.$ Then
\begin{align*}
F_n(x) &= \sum_{i=0}^n\binom{2n-i}{i}x^{n-i}
\end{align*}
and
\begin{align*}
G_n(x) &= \sum_{i=0}^n\binom{2n+1-i}{i}x^{n+1-i}.
\end{align*}
\end{lemma}

\begin{proof}
Since $L_u=\begin{bmatrix} 1 & 0\\ u & 1\end{bmatrix}$, it is clear that $F_0(x)=1$ and $G_0(x)=x$, which satisfy the desired conclusion in the case $n=0$. For $n\geq 0$, note that, by Proposition~\ref{alternating}, $F_{n+1}(x)=F_n(x)+G_n(x)$ and $G_{n+1}(x)=xF_n(x)+(1+x)G_n(x) = xF_{n+1}(x)+G_n(x)$. In particular, if we assume that the conclusion holds for some $k\geq 0$, then
by Lemma~\ref{aux} we obtain that
\begin{align*}
F_{k+1}(x) &= F_k(x)+G_k(x)\\
&= \sum_{i=0}^k\binom{2k-i}{i}x^{k-i} + \sum_{i=0}^k\binom{2k+1-i}{i}x^{k+1-i}\\
&= \sum_{i=0}^{k}\binom{2k+2-i}{i}x^{k+1-i}+1\\
&= \sum_{i=0}^{k+1}\binom{2k+2-i}{i}x^{k+1-i}.
\end{align*}
Also,
\begin{align*}
G_{k+1}(x) &= G_k(x)+xF_{k+1}(x)\\
&= \sum_{i=0}^k\binom{2k+1-i}{i}x^{k+1-i} + \sum_{i=0}^{k+1}\binom{2k+2-i}{i}x^{k+2-i}\\
&= \sum_{i=0}^{k+1}\binom{2k+3-i}{i}x^{k+2-i}.
\end{align*}
The result follows by induction.
\end{proof}

Note that $F_n(x^2)=\mathcal{F}_{2n-1}(x)$ where $\mathcal{F}_m(x)$ is the $m^{\text{th}}$ Fibonacci polynomial~\cite{BQ}.

\begin{lemma}\label{HnIn}
For any $n\geq 1$, let $H_n(x)$ and $I_n(x)$ be the polynomials over $\mathbb{N}_0$ such that $(R_1L_u)^nL_u=\begin{bmatrix} H_n(u) & \ast\\ I_n(u) & \ast\end{bmatrix}.$ Then
\begin{align*}
H_n(x) &= \sum_{i=0}^n\left(\binom{2n-i}{i}+\binom{2n-1-i}{i}\right)x^{n-i}
\end{align*}
and
\begin{align*}
I_n(x) &= \sum_{i=0}^{n-1}\left(\binom{2n-1-i}{i}+\binom{2n-2-i}{i}\right)x^{n-i}.
\end{align*}
\end{lemma}

\begin{proof}
As in Lemma~\ref{FnGn}, the case $n=1$ follows trivially. Note that $H_{n+1}(x)=(1+x)H_n(x)+I_n(x)$ and $I_{n+1}(x)=xH_n(x)+I_n(x)$. If we assume that the conclusion holds for some $k\geq 0$, then by Lemma~\ref{aux} we get that
\begin{align*}
I_{k+1}(x) &= xH_k(x)+I_k(x)\\
&= \sum_{i=0}^k\binom{2k-i}{i}x^{k+1-i}+\sum_{i=0}^{k-1}\binom{2k-1-i}{i}x^{k-i}+\sum_{i=0}^k\binom{2k-1-i}{i}x^{k+1-i}\\
& \qquad+\sum_{i=0}^{k-1}\binom{2k-2-i}{i}x^{k-i}\\
&=\sum_{i=0}^k\left(\binom{2k+1-i}{i}+\binom{2k-i}{i}\right)x^{k+1-i}
\end{align*}
and
\begin{align*}
H_{k+1}(x)&= H_k(x)+I_{k+1}(x)\\
&= \sum_{i=0}^k\left(\binom{2k-i}{i} x^{k-i}+\binom{2k+1-i}{i} x^{k+1-i}\right)\\
& \qquad+ \sum_{i=0}^k \left(\binom{2k-1-i}{i}x^{k-i}+\binom{2k-i}{i}x^{k+1-i}\right)\\
&= 1+\sum_{i=0}^{k}\binom{2k+2-i}{i}x^{k+1-i} + \sum_{i=0}^{k}\binom{2k+1-i}{i}x^{k+1-i}\\
&= \sum_{i=0}^{k+1}\left(\binom{2k+2-i}{i}+\binom{2k+1-i}{i}\right)x^{k+1-i}.
\end{align*}
The result follows by induction.
\end{proof}

The main difference between the cases $u\geq v\geq 2$ and (the current) $u\geq v=1$ is expressed by Lemma~\ref{HnIn} above. The failure of the inequality $2v\leq 1+uv$ in the proof of Proposition~\ref{ineq} means that we must consider two sets of families of matrices as candidates for the maximal left column entry of odd depth. While a little more work is involved, we obtain the desired result with the propositions that follow.

\begin{definition}
If $f(x)$ is a polynomial over $\mathbb{N}_0$, we let $[f]_n$ denote the coefficient of $f$ associated with $x^n$. If $n>\deg(f)$, then $[f]_n=0$.
\end{definition}

\begin{proposition}\label{sim}
For any $n\geq 1$, we have that:
\begin{enumerate}
\item[(a)] $I_n(x)\preccurlyeq H_n(x) \preccurlyeq G_n(x)$,
\item[(b)] $H_n(x) + I_n(x) \preccurlyeq F_n(x) + G_n(x)$.
\end{enumerate}
\end{proposition}

\begin{proof}
Since, for any $n\geq 1$, $(R_1L_u)^nL_u$ is $v$-UD, it follows that $I_n(x)\preccurlyeq H_n(x)$.  Let  $0\leq k\leq n$.  By Lemma~\ref{HnIn} and Lemma~\ref{aux} with $x=1$,

\[
\sum_{i\geq k}[H_n]_i = \sum_{i=0}^{n-k}\left(\binom{2n-i}{i}+\binom{2n-1-i}{i}\right)
= \sum_{i=0}^{n-k}\binom{2n+1-i}{i}+\binom{n+k-1}{n-k}
\]
and, by Lemma~\ref{FnGn},
\[
\sum_{i\geq k}[G_n]_i = \sum_{i=0}^{n-k+1}\binom{2n+1-i}{i}
= \sum_{i=0}^{n-k}\binom{2n+1-i}{i}+\binom{n+k}{n-k+1}.
\]
To complete the proof of (a), it is enough to show that $\binom{n+k-1}{n-k}\leq\binom{n+k}{n-k+1}$. Note that, for $k=0$, we have that the desired inequality holds trivially. For $k\ge 1$, since $n-k+1\leq n+k$,
\[
\binom{n+k-1}{n-k} \leq \binom{n+k-1}{n-k}\cdot\frac{n+k}{n-k+1}
= \binom{n+k}{n-k+1},
\]as desired.

By Lemma~\ref{aux} with $x=1$ and Lemma~\ref{HnIn},
\begin{align*}
\sum_{i\geq k}[H_n+I_n]_i &= \sum_{i=0}^{n-k}\left(\binom{2n-i}{i}+\binom{2n-1-i}{i}+\binom{2n-1-i}{i} +\binom{2n-2-i}{i}\right)\\
&= \sum_{i=0}^{n-k}\left(\binom{2n-i}{i}+\binom{2n+1-i}{i}\right)+\binom{n+k-2}{n-k}+\binom{n+k-1}{n-k}.
\end{align*}
As in the proof of (a), it can be shown that $\binom{n+k-2}{n-k}\leq\binom{n+k-1}{n-k+1}$ for $0\leq k\leq n$. This is enough to obtain (b) since, by Lemma~\ref{FnGn},
\begin{align*}
\sum_{i\geq k}[F_n+G_n]_i &= \sum_{i=0}^{n-k}\left(\binom{2n-i}{i}+\binom{2n+1-i}{i}\right)+\binom{n+k}{n-k+1}.
\end{align*}
\end{proof}

\begin{proposition}\label{sim2}
For any $n\geq 1$, we have that:
\begin{enumerate}
\item[(a)] $2xH_n(x)+I_n(x)\preccurlyeq G_{n+1}(x)$,
\item[(b)] $F_n(x) + 2G_n(x) \preccurlyeq H_{n+1}(x)$,
\item[(c)] $xF_n(x)+G_n(x) = I_{n+1}(x)$.
\end{enumerate}
\end{proposition}

\begin{proof}
By Lemma~\ref{aux} with $x=1$, Lemma~\ref{FnGn} and Lemma~\ref{HnIn}, for $0\leq k\leq n$, we have that
\begin{align*}
\sum_{i\geq k}[2xH_n+I_n]_i &= \sum_{i\geq k}[xH_n+I_{n+1}]_i\\
&= \sum_{i=0}^{n-k}\left(\binom{2n-1-i}{i}+2\binom{2n-i}{i}+\binom{2n+1-i}{i}\right)\\
&= \sum_{i=0}^{n-k+1}\left(\binom{2n+1-i}{i}+\binom{2n+2-i}{i}\right)\\
&\qquad-\binom{n+k-1}{n-k+1}-\binom{n+k}{n-k+1}\\
&= \sum_{i=0}^{n-k+2}\binom{2n+3-i}{i}-\binom{n+k-1}{n-k+1}\\
&\qquad-\binom{n+k}{n-k+1}-\binom{n+k}{n-k+2}\\
&= \sum_{i=0}^{n-k+2}\binom{2n+3-i}{i}-\binom{n+k-1}{n-k+1}-\binom{n+k+1}{n-k+2}\\
&\leq \sum_{i=0}^{n-k+2}\binom{2n+3-i}{i}\\
&= \sum_{i\geq k}[G_{n+1}]_i,
\end{align*}
proving (a).

By Lemma~\ref{aux} with $x=1$, Lemma~\ref{FnGn} and Lemma~\ref{HnIn}, for $0\leq k\leq n$, we have that
\begin{align*}
\sum_{i\geq k}[F_n+2G_n]_i &= \sum_{i=0}^{n-k}\left(\binom{2n-i}{i}+2\binom{2n+1-i}{i}\right)+2\binom{n+k}{n-k+1}\\
&= \sum_{i=0}^{n-k}\left(\binom{2n+2-i}{i}+\binom{2n+1-i}{i}\right)+\binom{n+k}{n-k}\\
&\qquad+2\binom{n+k}{n-k+1}\\
&= \sum_{i=0}^{n-k}\left(\binom{2n+2-i}{i}+\binom{2n+1-i}{i}\right)+\binom{n+k+1}{n-k+1}\\
&\qquad+\binom{n+k}{n-k+1}\\
&= \sum_{i=0}^{n-k+1}\left(\binom{2n+2-i}{i}+\binom{2n+1-i}{i}\right)\\
&= \sum_{i\geq k}[H_{n+1}]_i,
\end{align*}
which gives (b).

Part (c) follows quickly from Lemma~\ref{FnGn} and Lemma~\ref{HnIn}:
\begin{align*}
xF_n(x)+G_n(x) &= \sum_{i=0}^n\binom{2n-i}{i}x^{n+1-i}+\sum_{i=0}^n\binom{2n+1-i}{i}x^{n+1-i}\\
&= \sum_{i=0}^{n}\left(\binom{2n+1-i}{i}+\binom{2n-i}{i}\right)x^{n+1-i}\\
&= I_{n+1}(x).
\end{align*}
\end{proof}

\begin{proposition}\label{v1version}
Suppose that $M\in \mathcal{T}^{(u,1)}(I_2,2n+1)$ is given by $M=\begin{bmatrix} a & \ast\\ c & \ast\end{bmatrix}$. Then $\max\{a,c\}\leq C_n$ and $a'+c'\leq A_n+C_n$, where $A_n$ and $C_n$ are as defined in Proposition~\ref{alternating}.
\end{proposition}

\begin{proof}
By Lemma~\ref{polyrep} we have that, for any $n$, $a=f(u)$ and $c=g(u)$ for some polynomials $f(x)$ and $g(x)$ over $\mathbb{N}_0$. By Lemma~\ref{domineq} and Proposition~\ref{sim}, to prove the proposition, it is enough to show that $f(x)\preccurlyeq F_n(x)$ and $g(x)\preccurlyeq G_n(x)$ if $M$ is $u$-LD and $g(x)\preccurlyeq I_n(x)$ and $f(x)\preccurlyeq H_n(x)$ if $M$ is $1$-UD.

As in the proof of Proposition~\ref{ineq}, the above claim is trivially true for $n=0$.

Suppose that the statement is true for all matrices of depth $2k+1$, for some $k\geq 0$. Let $M\in\mathcal{T}^{(u,v)}(I_2,2k+3)$. Then $M\in\mathcal{T}^{(u,v)}(M',2)$ for some $M'\in\mathcal{T}^{(u,v)}(I_2,2k+1)$ with $M'=\begin{bmatrix} \overline{f}(u) & \ast\\ \overline{g}(u) & \ast\end{bmatrix}$ for some polynomials $\overline{f}(x)$ and $\overline{g}(x)$ over $\mathbb{N}_0$. It follows that
\begin{align*}
M=\begin{cases}
\begin{bmatrix} \overline{f}(u) & \ast\\ 2u\overline{f}(u)+\overline{g}(u) & \ast\end{bmatrix} & \text{ if $M=L_u^2M'$,}\\[3ex]
\begin{bmatrix} \overline{f}(u)+\overline{g}(u) & \ast\\ u\overline{f}(u)+(1+u)\overline{g}(u) & \ast\end{bmatrix} & \text{ if $M=L_uR_1M'$,}\\[3ex]
\begin{bmatrix} (1+u)\overline{f}(u)+\overline{g}(u) & \ast\\ u\overline{f}(u)+\overline{g}(u) & \ast\end{bmatrix} & \text{ if $M=R_1L_uM'$,}\\[3ex]
\begin{bmatrix} \overline{f}(u)+2\overline{g}(u) & \ast\\ \overline{g}(u) & \ast\end{bmatrix} & \text{ if $M=R_1^2M'$.}
\end{cases}
\end{align*}
If $M'$ is $u$-LD, then $\overline{g}(x)\succcurlyeq x\overline{f}(x)$. Furthermore, by assumption, it follows that
\begin{align*}
\overline{f}(x) &\preccurlyeq \overline{f}(x)+\overline{g}(x)\\
&\preccurlyeq F_k(x)+G_k(x)\\
&= F_{k+1}(x)
\end{align*}
and
\begin{align*}
2x\overline{f}(x)+\overline{g}(x) &= x\overline{f}(x)+x\overline{f}(x)+\overline{g}(x)\\
&\preccurlyeq x\overline{f}(x)+\overline{g}(x)+\overline{g}(x)\\
&\preccurlyeq x\overline{f}(x)+(1+x)\overline{g}(x)\\
&\preccurlyeq xF_k(x)+(1+x)G_k(x)\\
&= G_{k+1}(x).
\end{align*}
This shows that our claim holds if $M$ is $u$-LD in this case.

By assumption and Proposition~\ref{sim2} part (b) and (c), we have that
\begin{align*}
(1+x)\overline{f}(x)+\overline{g}(x) &\preccurlyeq \overline{f}(x)+2\overline{g}(x)\\
&\preccurlyeq F_k(x)+2G_k(x)\\
&\preccurlyeq H_{k+1}(x)
\end{align*}
and
\begin{align*}
\overline{g}(x)&\preccurlyeq x\overline{f}(x)+\overline{g}(x)\\
&\preccurlyeq xF_k(x)+G_k(x)\\
&= I_{k+1}(x).
\end{align*}
This shows that our claim also holds if $M$ is $1$-UD in this case.

If $M'$ is $1$-UD, then $\overline{f}(x)\succcurlyeq \overline{g}(x)$. Furthermore, by assumption, Proposition~\ref{sim} parts (a) and (b), and Proposition~\ref{sim2} part (a), we have that
\begin{align*}
\overline{f}(x) &\preccurlyeq \overline{f}(x)+\overline{g}(x)\\
&\preccurlyeq H_k(x)+I_k(x)\\
&\preccurlyeq F_k(x)+G_k(x)\\
&= F_{k+1}(x),
\end{align*}
\begin{align*}
2x\overline{f}(x)+\overline{g}(x) &\preccurlyeq 2xH_k(x)+I_k(x)\\
&\preccurlyeq G_{k+1}(x),
\end{align*}
and
\begin{align*}
x\overline{f}(x)+(1+x)\overline{g}(x) &\preccurlyeq xH_k(x)+(1+x)I_k(x)\\
&\preccurlyeq G_{k+1}(x).
\end{align*}
This shows that our claim holds if $M$ is $u$-LD in this case.

Finally,
\begin{align*}
\overline{f}(x)+2\overline{g}(x) &\preccurlyeq (1+x)\overline{f}(x)+\overline{g}(x)\\
 &\preccurlyeq (1+x)H_k(x)+I_k(x)\\
&= H_{k+1}(x)
\end{align*}
and
\begin{align*}
\overline{g}(x) &\preccurlyeq x\overline{f}(x)+\overline{g}(x)\\
&\preccurlyeq xH_k(x)+I_k(x)\\
&= I_{k+1}(x).
\end{align*}
This shows that our claim also holds if $M$ is $1$-UD in this case.
\end{proof}

Proposition~\ref{ineq} and Proposition~\ref{v1version} show that, for $u\geq v$, the left column entries of any descendant of $L_u$ of depth $2n+1$ are bounded above by $C_n$. Furthermore, the propositions show that the upper bound is achieved by the $(2,1)$ entry of $(L_uR_v)^nL_u$. To complete the proof of~\eqref{oddmax} we must show that:
\begin{enumerate}
\item[(A)] the right column entries of any descendant of $L_u$ of depth $2n+1$ and
\item[(B)] all entries of any descendant of $R_v$  of depth $2n+1$
\end{enumerate}
are bounded above by $C_n$.

A proof by induction of (A) follows quickly by noticing that the right column entries of any descendant $M$ of $L_u$ (including $L_u$ itself) are bounded above by the corresponding left column entries of $M$ (see Figure~\ref{fig:idexample}). In fact, the same argument generalizes in the following way.

\begin{lemma}\label{domcol}
Let  $M=\begin{bmatrix}a & b \\ c & d\end{bmatrix}$. If $M$ is a vertex in $\mathcal{L}^{(u,v)}$ ($M$ is a vertex in $\mathcal{R}^{(u,v)}$), then $\mu(M)=\max\{a,c\}$ ($\mu(M)=\max\{b,d\}$).
\end{lemma}

In the case where $M=(R_vL_u)^n$ or $M=(L_uR_v)^nL_u$ for some $n\geq 0$, we can deduce Lemma~\ref{domcol} directly by noting the following relationship between the entries in the first column with the entries in the second column.

\begin{proposition}\label{Symm}
Let $n\geq
0$. If $M=(R_vL_u)^n$, then $M =\begin{bmatrix}bu+d & b \\ \frac{bu}{v} & d\end{bmatrix}$ for some integers $b$ and $d$.  If $M=(L_uR_v)^nL_u$, then $M =\begin{bmatrix} d & b \\ \frac{u}{v}(b+vd) & d\end{bmatrix}$ for some integers $b$ and $d$.
\end{proposition}

\begin{proof}
Suppose $M=(R_vL_u)^n$. The result is true for $n=0$. Suppose that $n=k$ and $M =\begin{bmatrix}bu+d & b \\ \frac{bu}{v} & d\end{bmatrix}$, then
\begin{align*}
    (R_vL_u)^{k+1} &= R_vL_uM\\
    &= \begin{bmatrix}1+uv & v \\ u & 1\end{bmatrix}\begin{bmatrix}bu+d & b \\ \frac{bu}{v} & d\end{bmatrix}\\
    &= \begin{bmatrix}(1+uv)(bu+d)+bu & b(1+uv)+vd \\ u(bu+d)+\frac{bu}{v} & bu+d\end{bmatrix}.
\end{align*}
It is easy to see that this matrix has the desired form, proving the first part of the proposition.

Suppose $M=(L_uR_v)^nL_u$. The result is true for $n=0$, where $M=\begin{bmatrix}1 & 0 \\ u & 1\end{bmatrix}$. Suppose that $n=k$ and $M =\begin{bmatrix}d & b \\ \frac{u}{v}(b+vd) & d\end{bmatrix}$, then
\begin{align*}
    (L_uR_v)^{k+1}L_u &= L_uR_vM\\
    &=\begin{bmatrix}1 & v \\ u & 1+uv\end{bmatrix}\begin{bmatrix}d & b \\ \frac{u}{v}(b+vd) & d\end{bmatrix}\\
    &=\begin{bmatrix} d+u(b+vd) & b+vd \\ du+\frac{u}{v}(1+uv)(b+vd) & bu+d(1+uv)\end{bmatrix}.
\end{align*}
It is easy to see that this matrix has the desired form, proving the second part of the proposition.
\end{proof}

Given Proposition~\ref{Symm} and Corollary~\ref{flip} below, we can also show that $(L_uR_v)^n=\begin{bmatrix}d & b \\ \frac{bu}{v} & bu+d\end{bmatrix}$ and $(R_vL_u)^nR_v=\begin{bmatrix}d & b+vd \\ \frac{bu}{v} & d\end{bmatrix}$ for some integers $b$ and $d$.

It remains to prove (B).

\begin{proposition}\label{polyentry}
Let $M\in\mathcal{T}^{(u,v)}(I_2;n)$. Then
\begin{enumerate}
\item[(a)] $M =\begin{bmatrix}f_1(uv) & f_2(uv) \\ f_3(uv) & f_4(uv)\end{bmatrix}$ where $f_i(X)\in\mathbb{N}_0[X]$ and $\deg(f_i)\leq n$ for $i=1,2,3,4$.
\item[(b)] Futhermore,
\begin{align*}
    f_1(X) &=\sum_{i}a_iX^{\alpha_i},\\
    f_2(X) &=v\sum_{i}b_iX^{\beta_i},\\
    f_3(X) &=u\sum_{i}c_iX^{\gamma_i}\text{, and}\\
    f_4(X) &=\sum_{i}d_iX^{\delta_i}.
\end{align*}
\end{enumerate}
\end{proposition}

\begin{proof}
\begin{enumerate}
\item[(a)] The statement is clearly true in the case where $M=I_2$.

Suppose that the statement holds for all matrices in $\mathcal{T}^{(u,v)}(I_2;k)$ for some $k\geq 0$. Let $M\in\mathcal{T}^{(u,v)}(I_2;k+1)$. Then $M\in\{L_uM', R_vM'\}$ for some $M'\in\mathcal{T}^{(u,v)}(I_2;k)$. In particular, by assumption, we have that $$M' =\begin{bmatrix}f'_1(uv) & f'_2(uv) \\ f'_3(uv) & f'_4(uv)\end{bmatrix}$$ where $f'_i(X)\in\mathbb{N}_0[X]$ and $\deg(f'_i)\leq k$ for $i=1,2,3,4$. It now follows that
\begin{equation}\label{mleftright}
M=\begin{cases}
\begin{bmatrix}f'_1(uv) & f'_2(uv) \\ uf'_1(uv)+f'_3(uv) & uf'_2(uv)+f'_4(uv)\end{bmatrix} & \text{ if $M=L_uM'$,}\\[3ex]
\begin{bmatrix}f'_1(uv)+vf'_3(uv) & f'_2(uv)+vf'_4(uv) \\ f'_3(uv) & f'_4(uv)\end{bmatrix} & \text{ if $M=R_vM'$.}\\[3ex]
\end{cases}
\end{equation}
It is clear that, in either case, the statement holds for $M$ and therefore the result follows by induction.
\item[(b)]
The statement is clearly true in the case where $M=I_2$.

Suppose that the statement holds for all matrices in $\mathcal{T}^{(u,v)}(I_2;k)$ for some $k\geq 0$. Let $M\in\mathcal{T}^{(u,v)}(I_2;k+1)$. Then $M\in\{L_uM', R_vM'\}$ for some $M'\in\mathcal{T}^{(u,v)}(I_2;k)$. Suppose $M=L_uM'$. By assumption, we have that
\begin{align*}
    f'_1(X) &=\sum_{i}a_iX^{\alpha_i},\\
    f'_2(X) &=v\sum_{i}b_iX^{\beta_i},\\
    f'_3(X) &=u\sum_{i}c_iX^{\gamma_i}\text{, and}\\
    f'_4(X) &=\sum_{i}d_iX^{\delta_i}.
\end{align*}
Using~\eqref{mleftright}, it follows that
\begin{align*}
    f_1(X) &= f'_1(X)\\
    &=\sum_{i}a_iX^{\alpha_i},\\
    f_2(X) &= f'_2(X)\\
    &=v\sum_{i}b_iX^{\beta_i},\\
    f_3(X) &= uf'_1(X) + f'_3(X)\\
    &=u\sum_{i}a_iX^{\alpha_i}+u\sum_{i}c_iX^{\gamma_i}, \text{ and}
\end{align*}
\begin{align*}
    uf'_2(X)+f'_4(X) &=uv\sum_{i}b_iX^{\beta_i}+\sum_{i}d_iX^{\delta_i}.
\end{align*}
Note that since
\begin{align*}
    uf'_2(uv)+f'_4(uv) &=uv\sum_{i}b_i(uv)^{\beta_i}+\sum_{i}d_i(uv)^{\delta_i}\\
    &= \sum_{i}b_i(uv)^{\beta_i+1}+\sum_{i}d_i(uv)^{\delta_i},
\end{align*}
the statement holds with $f_4(X)=\sum_{i}b_iX^{\beta_i+1}+\sum_{i}d_iX^{\delta_i}$.

A similar argument applies in the case when $M=R_vM'$.

Having exhausted all possibilities, the statement holds for $M$ and therefore the result follows by induction.
\end{enumerate}
\end{proof}

%We denote by $c_{I_2}^{(u,v)}(n,i)$ the $i^{\text{th}}$ element (from left to right) of the $n^{\text{th}}$ row  in $\mathcal{T}^{(u,v)}(I_2)$.
The following proposition serves two purposes. It addresses the case $v>u$ by showing that $\mu(\mathcal{T}^{(u,v)}(I_2;n))=\mu(\mathcal{T}^{(v,u)}(I_2;n))$ and it is needed for the proof of (B).

\begin{proposition}\label{uvvu}
Let $n\geq 1$ and $i\in\{1,\ldots,2^n\}$. If $c_{I_2}^{(u,v)}(n,i)= \begin{bmatrix} a & b\\ c & d \end{bmatrix}$, then
$c_{I_2}^{(v,u)}(n,2^n+1-i)= \begin{bmatrix} d & c\\ b & a \end{bmatrix}$.
\end{proposition}
\begin{proof}
We have that
\[c_{I_2}^{(u,v)}(1,1) = L_u = \begin{bmatrix} 1 & 0\\ u & 1 \end{bmatrix}, \quad
c_{I_2}^{(v,u)}(1,2) = R_u = \begin{bmatrix} 1 & u\\ 0 & 1 \end{bmatrix}, \]
\[c_{I_2}^{(u,v)}(1,2) = R_v = \begin{bmatrix} 1 & v\\ 0 & 1 \end{bmatrix}, \quad
c_{I_2}^{(v,u)}(1,1) = L_v = \begin{bmatrix} 1 & 0\\ v & 1 \end{bmatrix}. \]
This shows that the result is true when $n=1$. Suppose that it is also true for all matrices in the $k^{\text{th}}$ row. Take an odd $i$ in $\{1,\ldots, 2^{k+1}\}$.   Assume that $c_{I_2}^{(u,v)}(k,(i+1)/2)=\begin{bmatrix}
a' & b' \\ c' & d' \end{bmatrix}$. Then
\begin{align*}
c_{I_2}^{(u,v)}(k+1,i) &= L_u\cdot c_{I_2}^{(u,v)}(k,(i+1)/2)\\
&= \begin{bmatrix} a' & b' \\ ua'+c' & ub'+d' \end{bmatrix}
\end{align*}
and
\begin{align*}
c_{I_2}^{(v,u)}(k+1,2^{k+1}+1-i)&= R_u\cdot c_{I_2}^{(v,u)}(k,(2^{k+1}+1-i)/2)\\
&= R_u\cdot  \begin{bmatrix} d' & c'\\ b' & a' \end{bmatrix}\\
&= \begin{bmatrix}  ub'+d' &  ua'+c' \\ b' & a' \end{bmatrix},
\end{align*}
since $2^{k+1}+1-i$ is even and $(2^{k+1}+1-i)/2=2^k+1-(i+1)/2$. When $i$ is even, the proof follows in a similar way. The result follows by induction.
\end{proof}

Let $M$ be a vertex in $\mathcal{R}^{(u,v)}$. By Proposition~\ref{uvvu}, there is a matrix $M'$ that is a vertex in $\mathcal{L}^{(v,u)}$ whose entries and depth are the same as $M$.
By Proposition~\ref{polyentry} part (a), the entries of $M'$ are polynomials in $u$ and $v$. Interchanging $uv$, we immediately obtain a relationship between the entries of matrices in $\mathcal{L}^{(u,v)}$ and $\mathcal{R}^{(u,v)}$ of the same depth. Corollary~\ref{flip} makes the above relationship precise (see Figure~\ref{fig:uvvu}).

\begin{corollary}\label{flip}
Let $n\geq 1$ and $i\in\{1,\ldots,2^n\}$. If $c_{I_2}^{(u,v)}(n,i)=\begin{bmatrix}f_1(uv) & f_2(uv) \\ f_3(uv) & f_4(uv)\end{bmatrix}$, then $c_{I_2}^{(u,v)}(n,2^n+1-i)=\begin{bmatrix}f_4(uv) & \frac{vf_3(uv)}{u} \\ \frac{uf_2(uv)}{v} & f_1(uv)\end{bmatrix}$.
\end{corollary}

\begin{figure}[ht!]
\begin{center}
\subfigure[The first three rows of $\mathcal{T}^{(u,v)}(I_2)$.]{
\begin{tikzpicture}[every tree node/.style={font=\small,anchor=base}, sibling distance=10pt]
\tikzset{level distance=35pt}
\Tree[.${\begin{bmatrix} 1 & 0\\ 0 & 1\end{bmatrix}}$ [.${\begin{bmatrix} 1 & 0\\ u & 1\end{bmatrix}}$ [.${\begin{bmatrix} 1 & 0\\ 2u & 1\end{bmatrix}}$ ]
   [.${\begin{bmatrix} 1+uv & v\\ u & 1\end{bmatrix}}$ ] ] [.${\begin{bmatrix} 1 & v\\ 0 & 1\end{bmatrix}}$ [.${\begin{bmatrix} 1 & v\\ u & 1+uv\end{bmatrix}}$ ]
   [.${\begin{bmatrix} 1 & 2v\\ 0 & 1\end{bmatrix}}$ ] ]]
\end{tikzpicture}}
\qquad
\subfigure[The first three rows of  $\mathcal{T}^{(v,u)}(I_2)$.]{
\begin{tikzpicture}[every tree node/.style={font=\small,anchor=base}, sibling distance=10pt]
\tikzset{level distance=35pt}
\Tree[.${\begin{bmatrix} 1 & 0\\ 0 & 1\end{bmatrix}}$ [.${\begin{bmatrix} 1 & 0\\ v & 1\end{bmatrix}}$ [.${\begin{bmatrix} 1 & 0\\ 2u & 1\end{bmatrix}}$ ]
   [.${\begin{bmatrix} 1+uv & u\\ v & 1\end{bmatrix}}$ ] ] [.${\begin{bmatrix} 1 & u\\ 0 & 1\end{bmatrix}}$ [.${\begin{bmatrix} 1 & u\\ v & 1+uv\end{bmatrix}}$ ]
   [.${\begin{bmatrix} 1 & 2u\\ 0 & 1\end{bmatrix}}$ ] ]]
\end{tikzpicture}}
\end{center}
\caption{A side-by-side comparison of the first three rows of $\mathcal{T}^{(u,v)}(I_2)$ and $\mathcal{T}^{(v,u)}(I_2)$.}\label{fig:uvvu}
\end{figure}

We are now in a position to prove (B).

\begin{proposition}\label{leftdom}
Let $M = c_{I_2}^{(u,v)}(n,i)$ for some $1\leq i\leq 2^{n-1}$ and\\ $M'=c_{I_2}^{(u,v)}(n,2^n+1-i)$ with $u\geq v$. Then $\mu(M')\leq \mu(M)$.
\end{proposition}

\begin{proof}
We have that $M =\begin{bmatrix}f_1(uv) & f_2(uv) \\ f_3(uv) & f_4(uv)\end{bmatrix}$ where $f_i(X)\in\mathbb{N}_0[X]$ for $i=1,2,3,4$ satisfy the conclusion of Proposition~\ref{polyentry}. By Corollary~\ref{flip}, $M'=\begin{bmatrix}f_4(uv) & \frac{vf_3(uv)}{u} \\ \frac{uf_2(uv)}{v} & f_1(uv)\end{bmatrix}$. By Lemma~\ref{domcol}, $\mu(M)=\max\{f_1(uv),f_3(uv)\}$ and $\mu(M')=\max\{f_1(uv),\frac{vf_3(uv)}{u}\}$.

If $M$ is $u$-LD, then $M'$ is $v$-UD. In particular,
\begin{align*}
\mu(M') &= \frac{vf_3(uv)}{u}\\
&= \frac{v}{u}\cdot u\sum_{i}c_i(uv)^{\gamma_i}\\
&\leq u\sum_{i}c_i(uv)^{\gamma_i}\\
&= f_3(uv)\\
&= \mu(M).
\end{align*}

If $M$ is $v$-UD, then $M'$ is $u$-LD. In particular,
\begin{align*}
\mu(M') &= f_1(uv)\\
&= \mu(M).
\end{align*}
\end{proof}

\begin{proof}[Proof of Theorem~\ref{main}]
The proofs of (A) and (B) using Lemma~\ref{domcol} and Proposition~\ref{leftdom}, respectively, complete the proof of~\eqref{oddmax} for all $u$ and $v$.

Applying Proposition~\ref{ineq} to the matrix $\overline{R}_v=R_v\begin{bmatrix}0 & 1 \\ 1 & 0\end{bmatrix}$, we get that, for $n\geq 0$,
\begin{align*}
\mu(\mathcal{T}^{(u,v)}(R_v;2n+1)) &= \mu(\mathcal{T}^{u,v}(\overline{R}_v;2n+1))\\
&= \mu((L_uR_v)^nL_u\overline{R}_v)\\
&= \mu((L_uR_v)^{n+1})
\end{align*}
since right multiplication by $\begin{bmatrix}0 & 1 \\ 1 & 0\end{bmatrix}$ simply exchanges the columns of a matrix. Note that, by Proposition~\ref{leftdom}, $\mu((R_vL_u)^{n+1})=\mu((L_uR_v)^{n+1})$.

Suppose that there exists an $M\in\mathcal{T}^{(u,v)}(I_2;2n+2)$ with $\mu(M)>\mu((R_vL_u)^{n+1})$. Proposition~\ref{leftdom} shows that we can assume that $M$ is a vertex in $\mathcal{L}^{(u,v)}$, so there is an $i$ such that $1\leq i\leq 2^{2n+1}$ and $M=c_{I_2}^{(u,v)}(2n+2,i)$. Let $M'=c_{I_2}^{(u,v)}(2n+2,2^{2n+2}+1-i)$. (Note that $M'$ is a vertex in $\mathcal{R}^{(u,v)}$.) By Corollary~\ref{flip} and Proposition~\ref{leftdom}, we obtain a contradiction if $M$ is $v$-UD. Assuming that $M$ is $u$-LD, $M=\begin{bmatrix}a & \ast \\ ua+c & \ast\end{bmatrix}$ where $\begin{bmatrix}a & \ast \\ c & \ast\end{bmatrix}$, the parent of $M$, is some matrix satisfying the result of Proposition~\ref{ineq}. That is, either $a\leq A_n$ and $c\leq C_n$ (with $ua\leq c$) or $a\leq \frac{v}{u}C_n$ and $c\leq A_n$. In either case it follows that
\begin{align*}
    \mu(M) &= ua+c\\
    &\leq A_n+vC_n\\
    &= \mu((R_vL_u)^{n+1}),
\end{align*}
a clear contradition. Therefore, no such $M$ exists, completing the proof of~\eqref{evenmax} when $u\geq v>1$.

For $u\geq v = 1$,~\eqref{evenmax} follows from Proposition~\ref{v1version} and Proposition~\ref{domineq} part (c) since, for $n\geq 0$,
\begin{align*}
\mu(\mathcal{T}^{(u,1)}(I_2;2n+2)) &\leq uF_n(u)+G_n(u)\\
&= \mu(L_u(L_uR_1)^nL_u).
\end{align*}
Finally,~\eqref{evenmax} follows for $v>u$ using a similar argument to (A) and (B).
\end{proof}

\begin{landscape}
\appendix
\section{Examples of PLFT $(u,v)$-Calkin-Wilf trees}
\begin{figure}[ht!]
\begin{center}
\begin{tikzpicture}[every tree node/.style={font=\tiny,anchor=base}, sibling distance=-4pt]
\tikzset{level distance=60pt}
\Tree[.${\begin{bmatrix} \circled{\textbf{1}} & 0\\ 0 & 1\end{bmatrix}}$ [.${\begin{bmatrix} 1 & 0\\ \circled{\textbf{1}} & 1\end{bmatrix}}$ [.${\begin{bmatrix} 1 & 0\\ 2 & 1\end{bmatrix}}$ [.${\begin{bmatrix} 1 & 0\\ 3 & 1\end{bmatrix}}$ ${\begin{bmatrix} 1 & 0\\ 4 & 1\end{bmatrix}}$ ${\begin{bmatrix} 4 & 1\\ 3 & 1\end{bmatrix}}$ ] [.${\begin{bmatrix} 3 & 1\\ 2 & 1\end{bmatrix}}$ ${\begin{bmatrix} 3 & 1\\ 5 & 2\end{bmatrix}}$ ${\begin{bmatrix} 5 & 2\\ 2 & 1\end{bmatrix}}$ ] ]
   [.${\begin{bmatrix} \circled{\textbf{2}} & 1\\ 1 & 1\end{bmatrix}}$
   [.${\begin{bmatrix} 2 & 1\\ \circled{\textbf{3}} & 2\end{bmatrix}}$
   ${\begin{bmatrix} 2 & 1\\ 5 & 3\end{bmatrix}}$
   ${\begin{bmatrix} \circled{{\textbf{5}}} & 3\\ 3 & 2\end{bmatrix}}$ ]
   [.${\begin{bmatrix} 3 & 2\\ 1 & 1\end{bmatrix}}$ ${\begin{bmatrix} 3 & 2\\ 4 & 3\end{bmatrix}}$
   ${\begin{bmatrix} 4 & 3\\ 1 & 1\end{bmatrix}}$ ] ] ] [.${\begin{bmatrix} 1 & 1\\ 0 & 1\end{bmatrix}}$ [.${\begin{bmatrix} 1 & 1\\ 1 & 2\end{bmatrix}}$
   [.${\begin{bmatrix} 1 & 1\\ 2 & 3\end{bmatrix}}$ ${\begin{bmatrix} 1 & 1\\ 3 & 4\end{bmatrix}}$ ${\begin{bmatrix} 3 & 4\\ 2 & 3\end{bmatrix}}$ ]  [.${\begin{bmatrix} 2 & 3\\ 1 & 2\end{bmatrix}}$ ${\begin{bmatrix} 2 & 3\\ 3 & 5\end{bmatrix}}$ ${\begin{bmatrix} 3 & 5\\ 1 & 2\end{bmatrix}}$ ]  ]
   [.${\begin{bmatrix} 1 & 2\\ 0 & 1\end{bmatrix}}$ [.${\begin{bmatrix} 1 & 2\\ 1 & 3\end{bmatrix}}$
   ${\begin{bmatrix} 1 & 2\\ 2 & 5\end{bmatrix}}$ ${\begin{bmatrix} 2 & 5\\ 1 & 3\end{bmatrix}}$ ] [.${\begin{bmatrix} 1 & 3\\ 0 & 1\end{bmatrix}}$
   ${\begin{bmatrix} 1 & 3\\ 1 & 4\end{bmatrix}}$
   ${\begin{bmatrix} 1 & 4\\ 0 & 1\end{bmatrix}}$  ] ]  ] ]
\end{tikzpicture}
\end{center}
\caption{The first five rows of the $\mathcal{T}^{(1,1)}(I_2)$ tree.}\label{fig:idexample1}
\end{figure}
\end{landscape}

\begin{landscape}
\begin{figure}[ht!]
\begin{center}
\begin{tikzpicture}[every tree node/.style={font=\tiny,anchor=base}, sibling distance=-6pt]
\tikzset{level distance=60pt}
\Tree[.${\begin{bmatrix} \circled{\textbf{1}} & 0\\ 0 & 1\end{bmatrix}}$ [.${\begin{bmatrix} 1 & 0\\ \circled{\textbf{5}} & 1\end{bmatrix}}$ [.${\begin{bmatrix} 1 & 0\\ 10 & 1\end{bmatrix}}$ [.${\begin{bmatrix} 1 & 0\\ 15 & 1\end{bmatrix}}$ ${\begin{bmatrix} 1 & 0\\ 20 & 1\end{bmatrix}}$ ${\begin{bmatrix} 31 & 2\\ 15 & 1\end{bmatrix}}$ ] [.${\begin{bmatrix} 21 & 2\\ 10 & 1\end{bmatrix}}$ ${\begin{bmatrix} 21 & 2\\ 115 & 11\end{bmatrix}}$ ${\begin{bmatrix} 41 & 4\\ 10 & 1\end{bmatrix}}$ ] ]
   [.${\begin{bmatrix} {\circled{\textbf{11}}} & 2\\ 5 & 1\end{bmatrix}}$
   [.${\begin{bmatrix} 11 & 2\\ \circled{\textbf{60}} & 11\end{bmatrix}}$
   ${\begin{bmatrix} 11 & 2\\ 115 & 21\end{bmatrix}}$
   ${\begin{bmatrix} \circled{\textbf{131}} & 24\\ 60 & 11\end{bmatrix}}$ ]
   [.${\begin{bmatrix} 21 & 4\\ 5 & 1\end{bmatrix}}$ ${\begin{bmatrix} 21 & 4\\ 110 & 21\end{bmatrix}}$
   ${\begin{bmatrix} 31 & 6\\ 5 & 1\end{bmatrix}}$ ] ] ] [.${\begin{bmatrix} 1 & 2\\ 0 & 1\end{bmatrix}}$ [.${\begin{bmatrix} 1 & 2\\ 5 & {\textbf{11}}\end{bmatrix}}$
   [.${\begin{bmatrix} 1 & 2\\ 10 & 21\end{bmatrix}}$ ${\begin{bmatrix} 1 & 2\\ 15 & 31\end{bmatrix}}$
   ${\begin{bmatrix} 21 & 44\\ 10 & 21\end{bmatrix}}$ ] [.${\begin{bmatrix} 11 & 24\\ 5 & 11\end{bmatrix}}$ ${\begin{bmatrix} 11 & 24\\ 60 & {\textbf{131}}\end{bmatrix}}$
   ${\begin{bmatrix} 21 & 46\\ 5 & 11\end{bmatrix}}$ ] ]
   [.${\begin{bmatrix} 1 & 4\\ 0 & 1\end{bmatrix}}$ [.${\begin{bmatrix} 1 & 4\\ 5 & 21\end{bmatrix}}$
   ${\begin{bmatrix} 1 & 4\\ 10 & 41\end{bmatrix}}$
   ${\begin{bmatrix} 11 & 46\\ 5 & 21\end{bmatrix}}$ ]
   [.${\begin{bmatrix} 1 & 6\\ 0 & 1\end{bmatrix}}$
   ${\begin{bmatrix} 1 & 6\\ 5 & 31\end{bmatrix}}$
   ${\begin{bmatrix} 1 & 8\\ 0 & 1\end{bmatrix}}$  ] ]  ] ]
\end{tikzpicture}
\end{center}
\caption{The first five rows of the $\mathcal{T}^{(5,2)}(I_2)$ tree.}\label{fig:idexample2}
\end{figure}
\end{landscape}

\begin{landscape}
\begin{figure}[ht!]
\begin{center}
\begin{tikzpicture}[every tree node/.style={font=\tiny,anchor=base}, sibling distance=-6pt]
\tikzset{level distance=60pt}
\Tree[.${\begin{bmatrix} 1& 0\\ 0 & \circled{\textbf{1}} \end{bmatrix}}$ [.${\begin{bmatrix} 1 & 0\\ 2 & 1\end{bmatrix}}$ [.${\begin{bmatrix} 1 & 0\\ 4 & 1\end{bmatrix}}$ [.${\begin{bmatrix} 1 & 0\\ 6 & 1\end{bmatrix}}$ ${\begin{bmatrix} 1 & 0\\ 8 & 1\end{bmatrix}}$ ${\begin{bmatrix} 31 & 5\\ 6 & 1\end{bmatrix}}$ ] [.${\begin{bmatrix} 21 & 5\\ 4 & 1\end{bmatrix}}$ ${\begin{bmatrix} 21 & 5\\ 46 & 11\end{bmatrix}}$ ${\begin{bmatrix} 41 & 10\\ 4 & 1\end{bmatrix}}$ ] ]
   [.${\begin{bmatrix} 11 & 5\\ 2 & 1\end{bmatrix}}$
   [.${\begin{bmatrix} 11 & 5\\ 24 & 11\end{bmatrix}}$
   ${\begin{bmatrix} 11 & 5 \\ 46 & 21\end{bmatrix}}$
   ${\begin{bmatrix} 131 & 60\\ 24 & 11\end{bmatrix}}$ ]
   [.${\begin{bmatrix} 21 & 10\\ 2 & 1\end{bmatrix}}$ ${\begin{bmatrix} 21 & 10\\ 44 & 21\end{bmatrix}}$
   ${\begin{bmatrix} 31 & 15\\ 2 & 1\end{bmatrix}}$ ] ] ] [.${\begin{bmatrix} 1 & \circled{\textbf{5}}\\ 0 & 1\end{bmatrix}}$ [.${\begin{bmatrix} 1 & 5\\ 2 & \circled{\textbf{11}}\end{bmatrix}}$
   [.${\begin{bmatrix} 1 & 5\\ 4 & 21\end{bmatrix}}$ ${\begin{bmatrix} 1 & 5\\ 6 & 31\end{bmatrix}}$
   ${\begin{bmatrix} 21 & 110\\ 4 & 21\end{bmatrix}}$ ]
   [.${\begin{bmatrix} 11 & \circled{\textbf{60}}\\ 2 & 11\end{bmatrix}}$ ${\begin{bmatrix} 11 & 60\\ 24 & \circled{\textbf{131}}\end{bmatrix}}$
   ${\begin{bmatrix} 21 & 115\\ 2 & 11\end{bmatrix}}$ ] ]
   [.${\begin{bmatrix} 1 & 10\\ 0 & 1\end{bmatrix}}$ [.${\begin{bmatrix} 1 & 10\\ 2 & 21\end{bmatrix}}$
   ${\begin{bmatrix} 1 & 10\\ 4 & 41\end{bmatrix}}$
   ${\begin{bmatrix} 11 & 115\\ 2 & 21\end{bmatrix}}$ ]
   [.${\begin{bmatrix} 1 & 15\\ 0 & 1\end{bmatrix}}$
   ${\begin{bmatrix} 1 & 15\\ 2 & 31\end{bmatrix}}$
   ${\begin{bmatrix} 1 & 20\\ 0 & 1\end{bmatrix}}$  ] ]  ] ]
\end{tikzpicture}
\end{center}
\caption{The first five rows of the  $\mathcal{T}^{(2,5)}(I_2)$ tree.}\label{fig:idexample3}
\end{figure}
\end{landscape}

\begin{landscape}
\begin{figure}[ht!]
\begin{center}
\begin{tikzpicture}[every tree node/.style={font=\tiny,anchor=base}, sibling distance=-6pt]
\tikzset{level distance=60pt}
\Tree[.${\begin{bmatrix} 1& 0\\ 0 & 1 \end{bmatrix}}$ [.${\begin{bmatrix} 1 & 0\\ \circled{\textbf{5}} & 1\end{bmatrix}}$ [.${\begin{bmatrix} 1 & 0\\ \circled{\textbf{10}} & 1\end{bmatrix}}$ [.${\begin{bmatrix} 1 & 0\\ 15 & 1\end{bmatrix}}$ ${\begin{bmatrix} 1 & 0\\ 20 & 1\end{bmatrix}}$ ${\begin{bmatrix} 16 & 1\\ 15 & 1\end{bmatrix}}$ ] [.${\begin{bmatrix} 11 & 1\\ 10 & 1\end{bmatrix}}$ ${\begin{bmatrix} 11 & 1\\ {\textbf{65}} & 6\end{bmatrix}}$ ${\begin{bmatrix} 21 & 2\\ 10 & 1\end{bmatrix}}$ ] ]
   [.${\begin{bmatrix} 6 & 1\\ 5 & 1\end{bmatrix}}$
   [.${\begin{bmatrix} 6 & 1\\ \circled{\textbf{35}} & 6\end{bmatrix}}$
   ${\begin{bmatrix} 6 & 1 \\ \circled{\textbf{65}} & 11\end{bmatrix}}$
   ${\begin{bmatrix} 41 & 7\\ 35 & 6\end{bmatrix}}$ ]
   [.${\begin{bmatrix} 11 & 2\\ 5 & 1\end{bmatrix}}$ ${\begin{bmatrix} 11 & 2\\ 60 & 11\end{bmatrix}}$
   ${\begin{bmatrix} 16 & 3\\ 5 & 1\end{bmatrix}}$ ] ] ] [.${\begin{bmatrix} 1 & 1\\ 0 & 1\end{bmatrix}}$ [.${\begin{bmatrix} 1 & 1\\ 5 & 6\end{bmatrix}}$
   [.${\begin{bmatrix} 1 & 1\\ 10 & 11\end{bmatrix}}$ ${\begin{bmatrix} 1 & 1\\ 15 & 16\end{bmatrix}}$
   ${\begin{bmatrix} 11 & 12\\ 10 & 11\end{bmatrix}}$ ]
   [.${\begin{bmatrix} 6 & 7\\ 5 & 6\end{bmatrix}}$ ${\begin{bmatrix} 6 & 7\\ 35 & 41\end{bmatrix}}$
   ${\begin{bmatrix} 11 & 13\\ 5 & 6\end{bmatrix}}$ ] ]
   [.${\begin{bmatrix} 1 & 2\\ 0 & 1\end{bmatrix}}$ [.${\begin{bmatrix} 1 & 2\\ 5 & 11\end{bmatrix}}$
   ${\begin{bmatrix} 1 & 2\\ 10 & 21\end{bmatrix}}$
   ${\begin{bmatrix} 6 & 13\\ 5 & 11\end{bmatrix}}$ ]
   [.${\begin{bmatrix} 1 & 3\\ 0 & 1\end{bmatrix}}$
   ${\begin{bmatrix} 1 & 3\\ 5 & 16\end{bmatrix}}$
   ${\begin{bmatrix} 1 & 4\\ 0 & 1\end{bmatrix}}$  ] ]  ] ]
\end{tikzpicture}
\end{center}
\caption{The first five rows of the  $\mathcal{T}^{(5,1)}(I_2)$ tree.}\label{fig:idexample4}
\end{figure}
\end{landscape}

%%%%%%%%%%%%%%%%%%%%%%%%%%%%%%%%%%%%%%%%%%%%%%%%%%


\begin{thebibliography}{99}

\bibitem{BBT1}
B. Bates, M. Bunder, and K. Tognetti, {\em Locating terms in the Stern-Brocot tree}, European J. Combin. {\bf 31} (2010), no. 7, 1020-1033.

\bibitem{BBT2}
B. Bates, M. Bunder, and K. Tognetti, {\em Linking the Calkin-Wilf and Stern-Brocot trees}, European J. Combin. {\bf 31} (2010), no. 7, 1637-1661.

\bibitem{BQ}
A.T. Benjamin and J.J. Quinn, {\em Proofs that Really Count: The Art of Combinatorial Proof}, The Dolciani Mathematical Expositions, {\bf 27}. Mathematical Association of America, Washington, DC, 2003. xiv+194 pp.

\bibitem{BG}
J. Bourgain and A. Gamburd, {\em Uniform expansion bounds for Cayley graphs of $SL_2(\mathbb F_p)$}, Ann. of Math. (2) {\bf 167} (2008), no. 2, 625-642.

\bibitem{B}
L. Bromberg, {\em Some applications of noncommutative groups and semigroups to information security}, Thesis (Ph.D.)–-City University of New York, 2015. 77 pp.

\bibitem{BSV}
L. Bromberg, V. Shpilrain, and A. Vdovina, {\em Navigating in the Cayley graph of $SL_2(\mathbb F_p)$ and applications to
hashing}, Semigroup Forum {\bf 94} (2017), no.2, 314-324.

\bibitem{CW}
N. Calkin and H.S. Wilf, {\em Recounting the rationals}, Amer. Math. Monthly {\bf 107} (2000), no. 4, 360-363.

\bibitem{HMST1}
S. Han, A.M. Masuda, S. Singh, and J. Thiel, {\em Orphans in forests of linear fractional transformations}, Electron. J. Combin.
{\bf 23} (2016), no. 3, Paper 3.6, 24pp.

\bibitem{HMST2}
S. Han, A.M. Masuda, S. Singh, and J. Thiel, {\em The (u,v)-Calkin-Wilf forest}, Int. J. Number Theory {\bf 12} (2016), no. 5,
1311-1328.

\bibitem{H}
H.A. Helfgott, {\em Growth and generation in $SL_2(\mathbb{Z}/p
\mathbb{Z})$}, Ann. of Math. (2) {\bf 167} (2008), no. 2, 601-623.

\bibitem{H1}
H.A. Helfgott, {\em Growth in groups: ideas and perspectives}, Bull. Amer. Math. Soc. (N.S.) {\bf{52}} (2015), no. 3, 357-413.

\bibitem{L}
E. Lucas, {\em Sur les suites de Farey}, Bull. Soc. Math. France, {\bf{6}} (1878), 118-119.

\bibitem{N1}
M.B. Nathanson, {\em Pairs of matrices in $GL_2(\mathbb{R}_{\geq 0})$ that freely generate}, Amer. Math. Monthly {\bf 122} (2015), no. 8, 790-792.

\bibitem{N2}
M.B. Nathanson, {\em A forest of linear fractional transformations}, Int. J. Number Theory {\bf 11} (2015), no. 4, 1275-1299.

\bibitem{P}
R. Paulin, {\em Largest values of the Stern sequence, alternating binary expansions and continuants}, J. Integer Seq. {\bf{20}} (2017), Article 17.2.8.

\bibitem{Z}
G. Z\'{e}mor, {\em Hash functions and graphs with large girths}, EUROCRYPT 91, Lecture Notes Comp. Sci. {\bf 547} (1991), 508-511.
\end{thebibliography}
\end{document}